\documentclass{amsart}
\usepackage{amsmath,amssymb, amsbsy}
\usepackage[dvips]{graphicx}
\usepackage[active]{srcltx}
\usepackage{color}
\newcommand{\R}{{\mathbb R}}
\newcommand{\ren}{\R^{N}}

\renewcommand{\l }{\lambda }

\newcommand{\p}{\partial}
\renewcommand{\O }{\Omega }

\newcommand{\inn}{\text{  in   }}
\newcommand{\dyle}{\displaystyle}
\newcommand{\dint}{\dyle\int}

\newcommand{\e }{\varepsilon}

\long\def\salta#1{\relax}

\newcommand{\re}{{I\!\!R}}

\newcommand{\irn}{\int_{\re^N}}
\newcommand{\io}{\int\limits_\O}

\newcommand{\Div}{\text{div\,}}

\renewcommand{\a }{\alpha }
\renewcommand{\b }{\beta }

\newcommand{\D }{\Delta }

\newcommand{\g }{\gamma}
\newcommand{\G }{\Gamma }
\renewcommand{\l }{\lambda }
\renewcommand{\L }{\Lambda }

\newcommand{\n }{\nabla }

\newcommand{\s }{\sigma }

\renewcommand{\O }{\Omega }

\renewcommand{\ge }{\geqslant}
\renewcommand{\geq }{\geqslant}
\renewcommand{\le }{\leqslant}
\renewcommand{\leq }{\leqslant}

\newtheorem{Theorem}{Theorem}[section]

\newtheorem{Lemma}[Theorem]{Lemma}

\theoremstyle{definition} 

\newtheorem{remark}[Theorem]{Remark}
\newcommand{\cqd}{{\unskip\nobreak\hfil\penalty50
        \hskip2em\hbox{}\nobreak\hfil\mbox{\rule{1ex}{1ex} \qquad}
        \parfillskip=0pt \finalhyphendemerits=0\par\medskip}}

\begin{document}

\title[]{Caffarelli-Kohn-Nirenberg type inequalities of fractional order with applications}

\author[B. Abdellaoui, R. Bentifour]{B. Abdellaoui, R. Bentifour}

\thanks{This work is partially supported by project
MTM2013-40846-P, MINECO, Spain. The first author is partially
supported by a grant form the ICTP centre of Trieste, Italy.}
\date{}


\address{Boumediene Abdellaoui, Rachid Bentifour
\hfill\break\indent Laboratoire d'Analyse Nonlin\'eaire et
Math\'ematiques Appliqu\'ees. \hfill \break\indent D\'epartement
de Math\'ematiques, Universit\'e Abou Bakr Belka\"{\i}d, Tlemcen,
\hfill\break\indent Tlemcen 13000, Algeria.} \email{{\tt
boumediene.abdellaoui@uam.es, rachidbentifour@gmail.com}}

 \begin{abstract} Let $0<s<1$ and $p>1$ be such that $ps<N$. Assume that $\Omega$ is a bounded domain containing the origin.
Starting from the ground state inequality by R. Frank and R.
Seiringer in \cite{FS} to obtain:
 \begin{enumerate}
 \item The following improved  Hardy inequality for  $p\ge 2$:

For all $q<p$, there exists a positive constant $C\equiv C(\Omega,
q, N, s)$ such that
 $$
\dint_{\re^N}\dint_{\re^N} \,
\dfrac{|u(x)-u(y)|^{p}}{|x-y|^{N+ps}}\,dx\,dy - \Lambda_{N,p,s}
\dint_{\re^N} \dfrac{|u(x)|^p}{|x|^{p}}\,dx\geq C
\dint_{\Omega}\dint_{\Omega}\dfrac{|u(x)-u(y)|^p}{|x-y|^{N+qs}}dxdy,
$$
for all $u \in \mathcal{C}_0^\infty(\ren)$. Here $\Lambda_{N,p,s}$
is the optimal constant in the Hardy inequality \eqref{hardy}.
\item Define $p^*_{s}=\frac{pN}{N-ps}$ and let
$\beta<\frac{N-ps}{2}$, then
\begin{equation*}
\dint\limits_{\re^N}\dint\limits_{\re^N}
\dfrac{|u(x)-u(y)|^p}{|x-y|^{N+ps}|x|^{\b}|y|^{\b}} \,dy\,dx\ge
S(N,p,s,\beta)\Big(\dint\limits_{\ren}
\dfrac{|u(x)|^{p^*_{s}}}{|x|^{2\beta\frac{p^*_s}{p}}}\,dx\Big)^{\frac{p}{p^*_{s}}},
\end{equation*} for all $u\in \mathcal{C}^\infty_0(\O)$ where
$S\equiv S(N,p,s,\beta)>0$. \item If $\beta\equiv \frac{N-ps}{2}$,
as a consequence of the improved Hardy inequality, we obtain that
for all $q<p$, there exists a positive constant $C(\O)$ such that
\begin{equation*}
\dint\limits_{\re^N}\dint\limits_{\re^N}
\dfrac{|u(x)-u(y)|^p}{|x-y|^{N+ps}|x|^{\b}|y|^{\b}} \,dy\,dx\ge
C(\O)\Big(\dint\limits_{\O} \dfrac{|u(x)|^{p^*_{s,q}}}{|x|^{2\b
\frac{p^*_{s,q}}{p}}}\,dx\Big)^{\frac{p}{p^*_{s,q}}},
\end{equation*} for all $u\in \mathcal{C}^\infty_0(\O)$ where $p^*_{s,q}=\frac{pN}{N-qs}$.
\end{enumerate}
Notice that the previous inequalities can be understood as the
fractional extension of the Callarelli-Kohn-Nirenberg inequalities
in \cite{CKN}.
 \end{abstract}
\maketitle
\section{Introduction}

In \cite{CKN} the authors proved the following result
\begin{Theorem}[Caffarelli-Kohn-Nirenberg]\label{th:CKN}
Let $p,q,r,\a,\b,\s$ and $a$ be real constants such that $p,q\ge
1,\,r>0,\, 0\le a\le 1$, and $$
\frac1p+\frac{\a}{N},\,\frac1q+\frac{\b}{N},\,\frac1r+\frac{m}{N}>0,$$
where $m=a\s+(1-a)\b$. Then there exists a positive constant $C$
such that for all $u\in \mathcal{C}^{\infty}_0(\ren)$ we have $$
\Big|\Big||x|^{m}u\Big|\Big|_{L^r(\ren)}\le C\Big|\Big||x|^{\a}|\n
u|\Big|\Big|^a_{L^p(\ren)}\Big|\Big||x|^{\b}u\Big|\Big|^{1-a}_{L^q(\ren)},$$
if and only if the following relations hold: $$
\frac{1}{r}+\frac{m}{N}=a\Big(\frac{1}{p}+\frac{\a-1}{N}\Big)+(1-a)
\Big(\frac{1}{q}+\frac{\b}{N}\Big),$$ with $$0\le \a-\s  \mbox{
if }a>0, $$ and $$\a-\s\le 1 \mbox{  if }a>0 \mbox{ and  }
\frac{1}{r}+\frac{m}{N}=\frac{1}{p}+\frac{\a-1}{N}. $$
\end{Theorem}
This class of inequalities are related to the following local
elliptic problem
\begin{equation}\label{locals}
-\Div(|x|^{-p\g}|\n u|^{p-2}\n u)=0. \end{equation} As a
consequence of Theorem \ref{th:CKN}, it follows that $|x|^{-\g}$,
with $\g<\frac{N-p}{p}$, is an admissible weight in the sense that
if $u$ is a weak positive supersolution to \eqref{locals}, then it
satisfies a weak Harnack inequality.

More precisely, there exists a positive constant $\kappa>1$ such
that for all $0<q<\kappa(p-1)$,
$$
\Big(\int_{B_{2\rho}(x_0)}u^q(x)|x|^{-p\g} dx\Big)^{\frac 1q} \le
C\inf_{B_{\rho}(x_0)}u, $$ where $B_{2\rho}(x_0)\subset\subset\O$,
and $C>0$ depends only on $B$.

We refer to \cite{FB}, \cite{JTO} and the references therein for a
complete discussion and the proof of the weak Harnack inequality.

Notice that even the classical Harnack inequality holds for
positive solution to \eqref{locals}.

One of the main tools to get the weak Harnack inequality is a
weighted Sobolev inequality that can obtained directly from
Theorem \ref{th:CKN}.

An alternative argument to get the Sobolev inequality is to prove
a weighted Hardy inequality as it was observed in \cite{MAZ}.

The main goal of this paper is to follow this approach in order to
get a nonlocal version of the Caffarelli-Kohn-Nirenberg
inequalities.

In \cite{FS}, the authors proved the following Hardy inequality
stating that for $p>1$ with $sp<N$ and for all $\phi\in
\mathcal{C}^\infty_0(\ren)$,
\begin{equation}\label{hardy}
\dint_{\re^N}\dint_{\re^N}
\dfrac{|\phi(x)-\phi(y)|^p}{|x-y|^{N+ps}}dxdy\ge \L_{N,p,s}\irn
\dfrac{|\phi(x)|^p}{|x|^{ps}}dx
\end{equation}
where the constant $\L_{N,p,s}$ is given by
\begin{equation}\label{LL}
\L_{N,p,s}=2\int_0^\infty
|1-\sigma^{-\gamma}|^{p-2}(1-\sigma^{-\gamma})\sigma^{N-1}
K(\sigma)
\end{equation}
and
$$
K(\s)=\dint\limits_{|y'|=1}\dfrac{dH^{n-1}(y')}{|x'-\s
y'|^{N+ps}}.
$$
In the same paper, setting
$$
h_s(u)\equiv \dint_{\re^N}\dint_{\re^N}
\dfrac{|u(x)-u(y)|^{p}}{|x-y|^{N+ps}}dx dy -\Lambda_{N,p,s}
\dint_{\re^N} \dfrac{|u(x)|^p}{|x|^{ps}} dx,
$$
they proved that for $p\ge 2$, there exists a positive constant
$C=C(p,N,s)$ such that for all $u\in \mathcal{C}_0^\infty(\ren)$,
if $v=|x|^{\frac{N-ps}{p}}u $, it holds
\begin{equation}\begin{array}{rcl}
h_s(u)&\geq&
C\dint_{\ren}\dint_{\ren}\dfrac{|v(x)-v(y)|^p}{|x-y|^{N+ps}}\dfrac{\,dx}{|x|^{\frac{N-ps}{2}}}\dfrac{\,dy}{|y|^{\frac{N-ps}{2}}}.
\end{array} \label{IngPR}
\end{equation}
The above inequality turns to be equality for $p=2$ with $C=1$.

As a consequence of  \eqref{IngPR}, we easily get that
$\L_{N,p,s}$ is never achieved.

For $p=2$, the authors in \cite{APP} proved the next result:
\begin{Theorem}\label{main000}
Let $N\geq 1$, $0<s<1$ and $N>2s$. Assume that $\Omega\subset
\mathbb{R}^N$ is a bounded domain, then for all $1<q<2$, there
exists a positive constant $C=C(\Omega, q, N, s)$ such that for
all $u\in \mathcal{C}_0^\infty(\Omega)$,
\begin{equation}\label{BBB}
a_{N,s}\int_{\mathbb{R}^{N}}\int_{\mathbb{R}^{N}} \,
\frac{|u(x)-u(y)|^{2}}{|x-y|^{N+2s}}\,dx\,dy - \Lambda_{N,2,s}
\int_{\mathbb{R}^{N}} \frac{|u(x)|^2}{|x|^{2s}}\,dx\geq C(\Omega,
q, N, s)
\int_{\Omega}\int_{\Omega}\frac{|u(x)-u(y)|^2}{|x-y|^{N+qs}}\,dx\,dy.
\end{equation}
\end{Theorem}

One of the main results of this work is to generalize Theorem \ref{main000} to the case $p>2$. More precisely we have
the next Theorem:
\begin{Theorem}\label{main}
Let $p>2$, $0<s<1$ and $N>ps$. Assume that $\Omega\subset
\mathbb{R}^N$ is a bounded domain, then for all $1<q<p$, there
exists a positive constant $C=C(\Omega, q, N, s)$ such that for
all $u\in \mathcal{C}_0^\infty(\Omega)$,
\begin{equation}\label{sara}
\int_{\mathbb{R}^{N}}\int_{\mathbb{R}^{N}} \,
\frac{|u(x)-u(y)|^{p}}{|x-y|^{N+ps}}\,dx\,dy - \Lambda_{N,p,s}
\int_{\mathbb{R}^{N}} \frac{|u(x)|^p}{|x|^{ps}}\,dx\geq C
\int_{\Omega}\int_{\Omega}\frac{|u(x)-u(y)|^p}{|x-y|^{N+qs}}\,dx\,dy.
\end{equation}
\end{Theorem}

As a consequence we get the next "fractional"
Caffarelli-Kohn-Nirenberg inequality in bounded domain.

\begin{Theorem}\label{main01}
Let $p\ge 2$, $0<s<1$ and $N>ps$. Assume that $\Omega\subset
\mathbb{R}^N$ is a bounded domain, then for all $1<q<p$, there
exists a positive constant $C=C(\Omega, q, N, s)$ such that for
all $u\in \mathcal{C}_0^\infty(\Omega)$,
\begin{equation}\label{sara00}
\dint_{\ren}\dint_{\ren}\dfrac{|u(x)-u(y)|^p}{|x-y|^{N+ps}}\dfrac{\,dx}{|x|^{\b}}\dfrac{\,dy}{|y|^{\b}}\,dx\,dy
\ge C\Big(\dint\limits_{\O} \dfrac{|u(x)|^{p^*_{s,q}}}{|x|^{2\b
\frac{p^*_{s,q}}{p}}}\,dx\Big)^{\frac{p}{p^*_{s,q}}}
\end{equation}
where $p^*_{s,q}=\frac{pN}{N-qs}$ and $\b=\frac{N-ps}{2}$.
\end{Theorem}

In the case where $\O=\ren$, to get a \emph{natural}
generalization of the classical Caffarelli-Kohn-Nirenberg
inequality obtained in \cite{CKN}, we have to consider a class of
\emph{admissible weights} in the sense of \cite{JTO}. Precisely
we  obtain the following weighted Sobolev inequality.
\begin{Theorem}\label{CKN01}
Assume that $1<p<\frac{N}{s}$ and let $0<\b<\frac{N-ps}{2},$ then
for all $u \in \mathcal{C}_0^\infty(\ren)$, we have
\begin{equation}\label{CKNNN}
\begin{array}{rcl}
\dint\limits_{\re^N}\dint\limits_{\re^N}
\dfrac{|u(x)-u(y)|^p}{|x-y|^{N+ps}}\dfrac{dx}{|x|^{\beta}}
\dfrac{dy}{|y|^\beta}\ge S(\b)\Big(\dint\limits_{\re^N}
\dfrac{|u(x)|^{p^*_s}}{|x|^{2\beta\frac{p^*_s}{p}}}\,dx\Big)^{\frac{p}{p^*_s}},
\end{array}
\end{equation}
where $S(\b)>0$.
\end{Theorem}

It is clear that the condition imposed on $\beta$ coincides in
some sense with definition of \emph{admissible} weight given in
\cite{JTO}. The proof of Theorem \ref{CKN01} is based on some
weighted Hardy inequality given below.

\

As a direct application of the previous results, we will consider
the problem
\begin{equation}\label{probb}
\left\{
\begin{array}{rcll}
L_{p,s}\, u-\l\dfrac{|u|^{p-2}u}{|x|^{ps}} &= & |u|^{q-1}u, \quad  u>0  & \hbox{ in  } \Omega,\\
u&=&0\quad & \hbox{  in  } \mathbb{R}^N\setminus\Omega,
\end{array}
\right.
\end{equation}
where $$ L_{s,p}\, u(x):=\mbox{ P.V. } \int_{\mathbb{R}^{N}} \,
\frac{|u(x)-u(y)|^{p-2}(u(x)-u(y))}{|x-y|^{N+ps}}\,dy, $$ $0<\l\le
\L_{N,p,s}$ and $q>0$.

In the local case, the problem is reduced to
\begin{equation}\label{int0}
\left\{
\begin{array}{rcll}
-\D_p\, u-\l\dfrac{|u|^{p-2}u}{|x|^{p}}& = & |u|^{q-1}u,\quad  u>0  &\hbox{ in  } \Omega,\\
u&=&0\quad & \hbox{  on  } \p\O.
\end{array}
\right.
\end{equation}
For $p=2$, the authors in \cite{BDT} proved that if $q>q_+(2)$,
then problem \eqref{int0} has no distributional supersolution,
however, if $q<q_+(2)$, there exists a positive supersolution,
with $q_+(2)=1+\frac{2}{\theta_1}$,
$\theta_1=\frac{N-2}{2}-\sqrt{\L_{N,2}-\l}$ and
$\L_{N,2}=\frac{(N-2)^2}{4}$, the classical Hardy constant.

The case $p\neq 2$ was considered in \cite{APV} where the same
alternative holds with $q_+(p)=p-1+\frac{p}{\theta_p}$ where
$\theta_p$ is the smallest solution to the equation
$$
\Xi(s)=(p-1)s^p-(N-p)s^{p-1}+\l.
$$
The fractional case with $p=2$ was studied in \cite{F} and
\cite{BMP}. The authors proved the same alternative with
$q_+(2,s)=1+\frac{2s}{\theta}$ where $\theta\equiv
\theta(\l,s,N)>0.$

Our goal is to extend the results of \cite{F} and \cite{BMP} to
the case $p\neq 2$.

The paper is organized as follows.

In Section \ref{sec1} we prove the main results, namely Theorems
\ref{main}, \ref{main01} and \ref{CKN01}.

The starting point will be the proof of a general version of the
Picone inequality. As a consequence, we get a weighted version of
the Hardy inequality for a class of "admissible weights".

Hence, following closely the arguments used in \cite{APP}, taking
in consideration the "weighted" Hardy inequality, we get the proof
of Theorem \ref{main}.

Once Theorem \ref{main} proved, we complete the proof of Theorem
\ref{main01} using suitable Sobolev inequality.

 At the end, and by using a weighted Hardy inequality, we are able
to get a "fractional Caffarelli-Kohn-Nirenberg" inequality for
admissible weights in $\ren$ and then to proof Theorem
\ref{CKN01}.

In section \ref{appl}, we analyze problem \eqref{int0}. We prove
the existence of a critical exponent $q_+(p,s)$ such that if
$q>q_+(p,s)$, then problem \eqref{int0} has no positive solution
in a suitable sense. To show the optimality of the non-existence
exponent, we will construct an appropriate supersolution in the
whole space.

In the whole of the paper we will use the next elementary
inequality, see for instance \cite{FS}.
\begin{Lemma}\label{alg}
Assume that $p>1$, then for all $0\le t\le 1$ and $a\in
\mathbb{C}$, we have
\begin{equation}\label{alg1}
|a-t|^p\geq(1-t)^{p-1}(|a|^p-t).
\end{equation}
\end{Lemma}

\section{Statement and proof of the main results}\label{sec1}
Let us begin with some functional settings that will be used
below, we refer to \cite{DPV} and \cite{MAZ} for more details.

For $s\in (0,1)$ and $p\ge 1$, we define the fractional Sobolev
spaces $W^{s,p}(\Omega)$, $\Omega\subset \ren$, by
$$
W^{s,p}(\Omega)\equiv \{u\in
L^p(\Omega):\dint_{\Omega}\dint_{\Omega}\dfrac{|u(x)-u(y)|^p}{|x-y|^{N+ps}}dxdy<\infty\}.
$$
It is clear that $W^{s,p}(\Omega)$ is a Banach space endowed with
the norm
$$
||u||_{W^{s,p}(\Omega)}=||u||_{L^p(\Omega)}+\Big(\dint_{\Omega}\dint_{\Omega}\dfrac{|u(x)-u(y)|^p}{|x-y|^{N+ps}}dxdy\Big)^{\frac
1p}.
$$
\

In the same way, we define the space $X^{s,p}_0(\O)$ as the
completion of $\mathcal{C}^\infty_0(\O)$ with respect to the norm
of $W^{s,p}(\O)$.

Notice that, if $Q=\ren\times \ren\setminus (\mathcal{C}\O\times
\mathcal{C}\O)$, then
$$
||\phi||_{X^{s,p}_0(\O)}=\Big(\dint\int_{Q}\dfrac{|\phi(x)-\phi(y)|^p}{|x-y|^{N+ps}}dxdy\Big)^{\frac
1p}+||\phi||_{L^p(\Omega)}.
$$
Using the fractional Sobolev inequality we obtain
$X^{s,p}_0(\O)\subset L^{p^*_s}(\O)$ with continuous inclusion,
where $p^*_s=\frac{pN}{N-sp}$ for $ps<N$.

In the case where $\Omega$ is a bounded regular domain, the space
$X^{s,p}_0(\O)$ can be endowed with the equivalent norm
$$
|||\phi|||_{X^{s,p}_0(\O)}=\Big(\dint\int_{Q}\dfrac{|\phi(x)-\phi(y)|^p}{|x-y|^{N+ps}}dxdy\Big)^{\frac
1p}.
$$

\

To prove the \emph{fractional} Caffarelli-Kohn-Nirenberg
inequality, we need to define fractional Sobolev spaces with
weight. More precisely, let $0< \beta<\frac{N-ps}{2}$ and
$\O\subset \ren$ with $0\in \O$, the weighted Sobolev space
$X^{s,p, \b}(\O)$ is defined by
$$
X^{s,p,\b}(\O)\dyle := \Big\{ \phi\in
L^p(\O,\frac{dx}{|x|^{2\beta}}):\dint_{\O}\dint_{\O}\dfrac{|\phi(x)-\phi(y)|^p}{|x-y|^{N+ps}}\dfrac{dxdy}{|x|^\beta|y|^\beta}<+\infty\Big\}.
$$
Thus $X^{s,p,\b}(\O)$ is a Banach space endowed with the norm
$$
\|\phi\|_{X^{s,p,\beta}(\O)}=
\Big(\dint_{\O}\frac{|\phi(x)|^pdx}{|x|^{2\beta}}\Big)^{\frac 1p}
+\Big(\dint_{\O}\dint_{\O}\dfrac{|\phi(x)-\phi(y)|^p}{|x-y|^{N+ps}}\dfrac{dxdy}{|x|^\beta|y|^\beta}\Big)^{\frac
1p}.
$$
Now, we define the weighted Sobolev space $X^{s,p,\beta}_0(\O)$ as
the completion of $\mathcal{C}^\infty_0(\O)$ with respect to the
previous norm.

As in \cite{Adams}, see also \cite{DPV}, we can prove the
following extension result.
\begin{Lemma}\label{ext}
Assume that $\Omega\subset \ren$ is a regular domain, then for all
$w\in X^{s,p,\beta}(\Omega)$, there exists $\tilde{w}\in
X^{s,p,\beta}(\ren)$ such that $\tilde{w}_{|\Omega}=w$ and
$$
||\tilde{w}||_{X^{s,p,\beta}(\ren)}\le C
||w||_{X^{s,p,\beta}(\Omega)}
$$
where $C\equiv C(N,s,p,\O)>0$.
\end{Lemma}

\begin{remark}\label{equiv}
As in the case $\beta=0$, if $\O$ is bounded regular domain, we
can endow $X^{s,p,\beta}_0(\O)$ with the equivalent norm
$$
|||\phi|||_{X^{s,p,\beta}_0(\O)}=
\Big(\dint_{\O}\dint_{\O}\dfrac{|\phi(x)-\phi(y)|^p}{|x-y|^{N+ps}}\dfrac{dxdy}{|x|^\beta|y|^\beta}\Big)^{\frac
1p}.
$$
\end{remark}
Now, for $w\in X^{s,p,\beta}(\ren)$, we set
$$
L_{s,p,\beta}(w)(x)=\mbox{ P.V. }
\dint_{\ren}\dfrac{|w(x)-w(y)|^{p-2}(w(x)-w(y))}{|x-y|^{N+ps}}\dfrac{dy}{|x|^\beta|y|^\beta}.
$$
It is clear that for all $w, v\in X^{s,p,\beta}(\ren)$, we have
$$
\langle L_{s,p,\beta}(w),v\rangle
=\dint_{\ren}\dint_{\ren}\dfrac{|w(x)-w(y)|^{p-2}(w(x)-w(y))(v(x)-v(y))}{|x-y|^{N+ps}}\dfrac{dxdy}{|x|^\beta|y|^\beta}.
$$
In the case where $\beta=0$, we denote $L_{s,p,\beta}$ by
$L_{s,p}$.

\

Let begin by proving the next version of the Picone inequality.
\begin{Lemma}(Picone inequality)\label{pic}
Let $w\in X^{s,p,\beta}_0(\O)$ be such that $w>0$ in $\O$. Assume
that $L_{s,p,\beta}(w)= \nu$ with $\nu\in L^1_{loc}(\ren)$ and
$\nu\gneqq 0$, then for all $u\in \mathcal{C}^\infty_0(\O)$, we
have
$$
\frac 12
\dint\dint_{Q}\dfrac{|u(x)-u(y)|^{p}}{|x-y|^{N+ps}}\dfrac{dxdy}{|x|^\beta|y|^\beta}\ge
\langle L_{s,p,\beta} w,\frac{|u|^p}{w^{p-1}}\rangle.
$$
\end{Lemma}
\begin{proof}
The case $\beta=0$ is obtained in \cite{PP} if $p=2$ and
in \cite{BPV} if $p\neq2$. For the reader convenience we include
some details for the general case $\beta\neq 0$.

We set $v(x)=\dfrac{|u(x)|^p}{|w(x)|^{p-1}}$ and
$k(x,y)=\dfrac{1}{|x-y|^{N+ps}|x|^{\beta}|y|^{\beta}}$, then
$$
\begin{array}{rcl}
\langle L_{s,p,\beta}(w(x)),v(x)\rangle&=&
\dint_{\O}v(x)\dint_{\ren}|w(x)-w(y)|^{p-2}(w(x)-w(y))k(x,y)
\,dy\,dx\\  \\
&=&\dint_{\O}\dfrac{|u(x)|^p}{|w(x)|^{p-1}}\dint_{\ren}|w(x)-w(y)|^{p-2}(w(x)-w(y))k(x,y) \,dy\,dx.\\
\end{array}
$$
Since $k$ is symmetric, we obtain that
$$
\begin{array}{rcl}
&\langle L_{s,p,\beta}(w(x)),v(x)\rangle=\\  \\
& \dfrac 12
\dint\dint_{Q}\left(\dfrac{|u(x)|^p}{|w(x)|^{p-1}}-\dfrac{|u(y)|^p}{|w(y)|^{p-1}}\right)|w(x)-w(y)|^{p-2}(w(x)-w(y))k(x,y) \,dy\,dx.\\
\end{array}
$$      ²²
Let $v_1=\dfrac{u}{w}$, then
$$
\begin{array}{rcl}
&\langle L_{s,p,\beta}(w(x)),v(x)\rangle=\\  \\
&\dfrac{1}{2}\dint\dint_{Q}\left(|v_1(x)|^p w(x)-|v_1(y)|^p
w(y)\right)|w(x)-w(y)|^{p-2}(w(x)-w(y))k(x,y) \,dy\,dx.
\end{array}
$$
Define $$\Phi(x,y)=|u(x)-u(y)|^p-\left(|v_1(x)|^p w(x)-|v_1(y)|^p
w(y)\right)|w(x)-w(y)|^{p-2}(w(x)-w(y)), $$ then
$$
\begin{array}{rcl}
&\langle
L_{s,p,\beta}(w(x)),v(x)\rangle+\dfrac{1}{2}\dint\limits_{Q}\Phi(x,y)k(x,y)
\,dy\,dx\\
&=\dfrac{1}{2}\dint\dint_{Q}|u(x)-u(y)|^p k(x,y) \,dy\,dx.
\end{array}
$$
We claim that $\Phi\geq 0$. It is clear that, by a symmetry
argument, we can assume that $w(x)\geq w(y)$. Let $t=w(y)/w(x),
a=u(x)/u(y)$, then using inequality \eqref{alg1}, the claim
follows at once. Hence we conclude.
\end{proof}

As a consequence, for $\beta=0$, we have the next comparison
principle that extends, to the fractional framework, the classical
one obtained by Brezis-Kamin in \cite{BK}.

\begin{Lemma}\label{compaa}
Let $\O$ be  a bounded domain and let $f$ be a nonnegative
continuous function such that $f(\s)>0$ if $\s>0$ and
$\dfrac{f(\s)}{\s^{p-1}}$ is decreasing. Let $u,v\in
W^{s,p}_0(\O)$ be such that $u,v>0$ in $\O$ and
$$\left\{
\begin{array}{rcl}
L_{s,p} u &\geq & f(u)\mbox{  in  }\O,\\
L_{s,p} v & \le & f(v)\mbox{  in   }\O,
\end{array}
\right.
$$
Then, $u\geq v\inn\Omega$.
\end{Lemma}
\begin{proof} Using an approximation argument, taking in consideration that $u,v>0$, we can prove that
\begin{equation}\label{nesr}
\dfrac{L_{s,p}u}{u^{p-1}}-\dfrac{L_{s,p}v}{v^{p-1}}\geq
\Big(\dfrac{f(u)}{u^{p-1}}-\dfrac{f(v)}{v^{p-1}}\Big).
\end{equation}
We set $\xi=(v^p-u^p)_{+}$, then
\begin{equation}\label{tlem}
\io(\dfrac{f(u)}{u^{p-1}}-\dfrac{f(v)}{v^{p-1}})\xi\,dx\leq
\int_{\O}\xi\Big(\dfrac{L_{s,p}u}{u^{p-1}}-\dfrac{L_{s,p}v}{v^{p-1}}\Big)\,dx.
\end{equation}
Let us analyze each term in the previous inequality.

Using the definition of $\xi$ we obtain that
$\Big(\dfrac{f(u)}{u^{p-1}}-\dfrac{f(v)}{v^{p-1}}\Big)\xi\ge 0$.
On the other hand, we have
\begin{eqnarray*}
&J\equiv
\dyle\io\xi\Big(\dfrac{L_{s,p}u}{u^{p-1}}-\dfrac{L_{s,p}v}{v^{p-1}}\Big)\,dx=\dfrac
12\dint\dint_{Q}
\dfrac{|u(x)-u(y)|^{p-2}(u(x)-u(y))}{|x-y|^{N+ps}}\Big(\dfrac{\xi(x)}{u^{p-1}(x)}-\dfrac{\xi(y)}{u^{p-1}(y)}\Big)\,dxdy\\
&-\dfrac 12\dint\dint_{Q} \dfrac{|v(x)-v(y)|^{p-2}(v(x)-v(y))}{|x-y|^{N+ps}}\Big(\dfrac{\xi(x)}{v^{p-1}(x)}-\dfrac{\xi(y)}{v^{p-1}(y)}\Big)\,dxdy,\\
\end{eqnarray*}
where $Q=\ren\times \ren\setminus (\mathcal{C}\O\times
\mathcal{C}\O$).

Notice that
$$
\begin{array}{lll}
&|u(x)-u(y)|^{p-2}(u(x)-u(y))\Big(\dfrac{\xi(x)}{u^{p-1}(x)}-\dfrac{\xi(y)}{u^{p-1}(y)}\Big)=\\
&|u(x)-u(y)|^{p-2}(u(x)-u(y))\Big(\dfrac{v^p(x)}{u^{p-1}(x)}-\dfrac{v^p(y)}{u^{p-1}(y)}\Big)\\
&-|u(x)-u(y)|^p.
\end{array}
$$
In the same way, we obtain that
$$
\begin{array}{lll}
&|v(x)-v(y)|^{p-2}(v(x)-v(y))\Big(\dfrac{\xi(x)}{v^{p-1}(x)}-\dfrac{\xi(y)}{v^{p-1}(y)}\Big)=\\
&-|v(x)-v(y)|^{p-2}(v(x)-v(y))\Big(\dfrac{u^p(x)}{v^{p-1}(x)}-\dfrac{u^p(y)}{v^{p-1}(y)}\Big)\\
&+|v(x)-v(y)|^p.
\end{array}
$$
Thus
$$
\begin{array}{lll}
J &=
&\dyle\dint\dint_{Q}\frac{|u(x)-u(y)|^{p-2}(u(x)-u(y))}{|x-y|^{N+ps}}\Big(\dfrac{v^p(x)}{u^{p-1}(x)}-\dfrac{v^p(y)}{u^{p-1}(y)}\Big)dxdy\\
&+&
\dyle\dint\dint_{Q}\frac{|v(x)-v(y)|^{p-2}(v(x)-v(y))}{|x-y|^{N+ps}}\Big(\dfrac{u^p(x)}{v^{p-1}(x)}-\dfrac{u^p(y)}{v^{p-1}(y)}\Big)dxdy\\
&-&
\dyle\dint\dint_{Q}\frac{|u(x)-u(y)|^p}{|x-y|^{N+ps}}-\dint\dint_{Q}\frac{|v(x)-v(y)|^p}{|x-y|^{N+ps}}dxdy\\
&=& \dyle \frac 12\io\dfrac{L_{p,s}(u)}{u^p}v^{p}dx+\frac
12\io\dfrac{L_{p,s}(v)}{v^{{p-1}}}u^p dx-
\dint\dint_{Q}\dfrac{|u(x)-u(y)|^p}{|x-y|^{N+ps}}dxdy-\dint\dint_{Q}\dfrac{|v(x)-v(y)|^p}{|x-y|^{N+ps}}dxdy.
\end{array}
$$
Now, using Picone's inequality, we conclude that $J\le 0$. Thus
$$
\Big(\dfrac{f(u)}{u^{p-1}}-\dfrac{f(v)}{v^{p-1}}\Big)\xi\equiv 0
$$
and then $\xi=0$ which implies that $u\le v$ in $\O$.
\end{proof}

\begin{remark}
The comparison result holds if we replace $f(s)$ by $g(s,x)$ where
$g$ is continuous in $s$ for a.e $x\in \O$, $\dfrac{g(s,x)}{s}$ is
decreasing for $s>0$ and $g(s,x)>0$ in $\O$ for all $s>0$ fixed.
\end{remark}

In the sequel we need the next results.
\begin{Lemma}\label{ictp1}
Fix $0<\beta<\frac{N-ps}{2}$ and let $w(x)=|x|^{-\g}$ with
$0<\gamma<\dfrac{N-ps-2\beta}{p-1}$, then there exists a positive
constant $\L(\g)>0$ such that
\begin{equation}\begin{array}{rcl}
L_{s,p,\beta}(w)= \L(\g) \dfrac{w^{p-1}}{|x|^{ps+2\beta}}\:\:\:
a.e \mbox{ in }\ren\backslash\{0\}.
\end{array} \label{EgL2}
\end{equation}
\end{Lemma}
\begin{proof} We set $r=|x|$ and $\rho=|y|$, then $x=rx', y=\rho y'$
where $|x'|=|y'|=1$. Thus
$$
\begin{array}{rcl}
L_{s,p,\beta}(w)&=& \dfrac{1}{|x|^{\beta}}
\dint\limits_0^{+\infty}|r^{-\gamma}-\rho^{-\gamma}|^{p-2}\dfrac{(r^{-\gamma}-\rho^{-\gamma})\rho^{N-1}}{\rho^{\beta}
r^{N+ps}}\left(
\dint\limits_{|y'|=1}\dfrac{dH^{n-1}(y')}{|x'-\frac{\rho}{r}
y'|^{N+ps}} \right) \,d\rho.
\end{array}
$$
Let $\sigma=\dfrac{\rho}{r}$, then
$$
L_{s,p,\beta}(w)=\dfrac{w^{p-1}(x)}{|x|^{ps+2\beta}}\dint\limits_0^{+\infty}
|1-\sigma^{-\gamma}|^{p-2}(1-\sigma^{-\gamma})\sigma^{N-\beta-1}
\left(\dint\limits_{|y'|=1}\dfrac{dH^{n-1}(y')}{|x'-\s y'|^{N+ps}}
\right) \,d\sigma.
$$
Defining
$$
K(\s)=\dint\limits_{|y'|=1}\dfrac{dH^{n-1}(y')}{|x'-\s
y'|^{N+ps}},
$$
as in \cite{FV}, we obtain that
\begin{equation}\label{kkk}
K(\sigma)=2\frac{\pi^{\frac{N-1}{2}}}{\Gamma(\frac{N-1}{2})}\int_0^\pi
\frac{\sin^{N-2}(\theta)}{(1-2\sigma \cos
(\theta)+\sigma^2)^{\frac{N+ps}{2}}}d\theta.
\end{equation}
Hence
$$
L_{s,p,\beta}(w)
=\dfrac{w^{p-1}(x)}{|x|^{ps+2\beta}}\dint\limits_0^{+\infty}
\psi(\sigma)\,d\sigma,
$$
with \begin{equation}\label{psi0}
 \psi(\sigma)=
|1-\sigma^{-\gamma}|^{p-2}(1-\sigma^{-\gamma})\sigma^{N-\beta-1}
K(\sigma). \end{equation} Define $\L(\g)\equiv
\dint\limits_0^{+\infty} \psi(\sigma)\,d\sigma$, then to finish we just have to show that $0<\L(\g)<\infty$.

We have
$$
\L(\gamma)=\int_0^1\psi(\sigma)\,d\sigma+\int_1^\infty
\psi(\sigma)\,d\sigma=I_1+I_2.
$$
Notice that $K(\frac{1}{\xi})=\xi^{N+ps}K(\xi)$ for any $\xi>0$,
then using the change of variable $\xi=\frac{1}{\sigma}$ in $I_1$,
there results that
\begin{equation}\label{ga}
\L(\g)=\dint\limits_1^{+\infty}K(\sigma)(\sigma^\g-1)^{p-1}\left(\sigma^{N-1-\beta-\g(p-1)}-\sigma^{\beta+ps-1}\right)\,d\sigma.
\end{equation}
As $\s\to \infty$, we have $$
K(\sigma)(\sigma^\g-1)^{p-1}\left(\sigma^{N-1-\beta-\g(p-1)}-\sigma^{\beta+ps-1}\right)\backsimeq
\s^{-1-\beta-ps}\in L^1(2,\infty).$$ Now, as, $\s\to 1$, we have
$$
K(\sigma)(\sigma^\g-1)^{p-1}\left(\sigma^{N-1-\beta-\g(p-1)}-\sigma^{\beta+ps-1}\right)\backsimeq
(\s-1)^{p-1-ps}\in L^1(1,2).$$ Therefore, combining the above
estimates, we get $|\L(\g)|<\infty$. Now, using the fact that
$0<\gamma<\dfrac{N-ps-2\beta}{p-1}$, from \eqref{ga}, we reach
that $\L(\g)>0$.

As a conclusion, we have proved that
$$
L_{s,p,\beta}(w)= \L(\g)\dfrac{w^{p-1}}{|x|^{ps+2\beta}}\:\:a.e.
\mbox{ in }\ren\backslash\{0\}.
$$
Hence the result follows.
\end{proof}
As a consequence we have the following weighted Hardy inequality.
\begin{Theorem}\label{HHH}
Let $\b<\frac{N-ps}{2},$ then for all $u \in
\mathcal{C}_0^\infty(\ren)$, we have
\begin{equation}
\begin{array}{rcl}
2\L(\g)\dint\limits_{\re^N} \dfrac{|u(x)|^p}{|x|^{ps+2\beta}}\,dx
&\leq&\dint\limits_{\re^N}\dint\limits_{\re^N}
\dfrac{|u(x)-u(y)|^p}{|x-y|^{N+ps}}\dfrac{dx}{|x|^{\beta}}
\dfrac{dy}{|y|^\beta},
\end{array} \label{IngL1}
\end{equation}
where $\L(\g)$ is defined in \eqref{ga}.
\end{Theorem}

\begin{proof}
Let $u \in \mathcal{C}_0^\infty(\ren)$ and $w(x)=|x|^{-\g}$ with
$\g<\dfrac{N-ps-2\beta}{p-1}$. By Lemma \ref{ictp1}, we have
$$
L_{p,s,\beta}(w)=\L(\g)\dfrac{w^{p-1}}{|x|^{ps+2\beta}}.
$$
It is clear that $\dfrac{w^{p-1}}{|x|^{ps+2\beta}}\in
L^1_{loc}(\ren)$. Thus using Picone inequality in Lemma \ref{pic},
it follows that
$$
\dfrac 12 \dint_{\ren}\dint_{\ren}
\dfrac{|u(x)-u(y)|^p}{|x-y|^{N+ps}}\dfrac{dx}{|x|^{\beta}}
\dfrac{dy}{|y|^{\beta}}\ge \langle L_{p,s,\b}
w,\frac{|u|^p}{w^{p-1}}\rangle=\L(\g)\dint_{\ren}
\dfrac{|u(x)|^p}{|x|^{ps+2\beta}}\,dx.$$ Thus we conclude.
\end{proof}
\begin{remark}
Let analyze the behavior of the constant $\L(\g)$ in inequality
\eqref{IngL1}. Recall that, for $\g<\dfrac{N-ps-2\beta}{p-1}$,
$$
\L(\g)=\dint\limits_1^{+\infty}K(\sigma)(\sigma^\g-1)^{p-1}\left(\sigma^{N-1-\beta-\g(p-1)}-\sigma^{\beta+ps-1}\right)\,d\sigma,
$$
then
$$
\L'(\g)=(p-1)\dint\limits_1^{+\infty}K(\sigma)\log(\s)(\sigma^\g-1)^{p-2}\left(\sigma^{N-1-\beta-\g(p-1)}-\sigma^{\beta+ps+\g-1}\right)\,d\sigma.
$$
It is clear that if $\g_0=\frac{N-\beta-ps}{p}$, then
$\L'(\g_0)=0$, $\L'(\g)>0$ if $\g<\g_0$ and $\L'(\g)<0$ if
$\g>\g_0$. Thus
$$
\max_{\{0<\g<\frac{N-ps-2\beta}{p-1}\}}\L(\g)=\L(\g_0).
$$
Hence
\begin{equation}\label{TRT00}
\dint_{\ren}\dint_{\ren}
\dfrac{|u(x)-u(y)|^p}{|x-y|^{N+ps}}\dfrac{dx}{|x|^{\beta}}
\dfrac{dy}{|y|^{\beta}}\ge 2\L(\g_0)\dint_{\ren}
\dfrac{|u(x)|^p}{|x|^{ps+2\beta}}\,dx.\end{equation} Notice that
for $\beta=0$, then $2\L(\g_0)=2\L(\frac{N-ps}{p})\equiv
\L_{N,p,s}$ given in \eqref{LL}. Therefore, we have the next
optimality result.
\end{remark}
\begin{Theorem}\label{YYY}
Define
$$ \L_{N,p,s,\g}=\inf_{\{\phi\in \mathcal{C}^\infty_0(\ren)\backslash
0\}}\dfrac{\dint_{\re^N}\dint_{\re^N}
\dfrac{|\phi(x)-\phi(y)|^p}{|x-y|^{N+ps}|x|^\beta|y|^\beta}dxdy}{\dyle\int_{\ren}\dfrac{|\phi(x)|^p}{|x|^{ps+2\beta}}dx},
$$
then $\L_{N,p,s,\g}=2\L(\g_0)$.
\end{Theorem}
\begin{proof}
From \eqref{TRT00}, it follows that $\L_{N,p,s,\g}\ge 2\L(\g_0)$,
hence to conclude we have just to prove the reverse inequality.

We closely follow the argument used in \cite{FS}.

Let $w_0(x)=|x|^{-\g_0}$, by Lemma \ref{ictp1}, we have
$$
L_{p,s,\beta}(w_0)=\L(\g_0)\dfrac{w^{p-1}_0}{|x|^{ps+2\beta}}.
$$
We set
$$
M_n=\{x\in \ren: 1\le |x|<n\}  \mbox{   and   }  O_n=\{x\in \ren:
 |x|\ge n\},$$
and define $$ w_n= \left\{\begin{array}{lll} 1-n^{-\g_0} & \mbox{
if
}& x\in B_1(0),\\
|x|^{-\g_0}-n^{-\g_0} &  \mbox{  if }& x\in M_n,\\
0 &  \mbox{  if }& x\in O_n.\\
\end{array}
\right.
$$
By a direct computation, we get easily that $w_n\in
X^{s,p,\beta}_0(\ren)$.

Hence
$$
\langle L_{p,s,\beta}(w_0), w_n\rangle =\L(\g_0)\irn
\dfrac{w_nw^{p-1}_0}{|x|^{ps+2\beta}}dx.
$$
Thus
$$
\dint_{\ren}\dint_{\ren}
\dfrac{(w_n(x)-w_n(y))|w_0(x)-w_0(y)|^{p-2}(w_0(x)-w_0(y))}{|x-y|^{N+ps}|x|^\beta
|y|^\beta}\,dxdy=2\L(\g_0)\irn
\dfrac{w_nw^{p-1}_0}{|x|^{ps+2\beta}}dx.
$$
Let analyze each term in the previous identity. As in \cite{FS} we
obtain that
$$
\begin{array}{lll}
\dint_{\ren}\dint_{\ren}
\dfrac{(w_n(x)-w_n(y))|w_0(x)-w_0(y)|^{p-2}(w_0(x)-w_0(y))}{|x-y|^{N+ps}|x|^\beta
|y|^\beta}\,dxdy\ge \\
\dint_{\ren}\dint_{\ren}
\dfrac{|w_n(x)-w_n(y)|^p}{|x-y|^{N+ps}|x|^\beta |y|^\beta}\,dxdy.
\end{array}
$$
On the other hand we have
$$
\irn \dfrac{w_nw^{p-1}_0}{|x|^{ps+2\beta}}dx=\irn
\dfrac{w^p_n}{|x|^{ps+2\beta}}dx+I_n+J_n,
$$
where
$$
I_n=\int_{B_1(0)}(1-n^{-\g_0})(w^{p-1}_0-(1-n^{-\g_0})^{p-1})\dfrac{dx}{|x|^{ps+\beta}},
$$
and
$$
J_n=\int_{M_n}(w_0(x)-n^{-\g_0})(w^{p-1}_0-(w_0(x)-n^{-\g_0})^{p-1})\dfrac{dx}{|x|^{ps+\beta}}.
$$
It is clear that $I_n,J_n\ge 0$, using a direct computation we can
prove that
$$
I_n+J_n\le C\mbox{  for all }n\ge 1.
$$
Thus, combining the above estimates, it holds
\begin{eqnarray}\label{KPL}
\dyle\L_{N,p,s,\g} &\le & \dfrac{\dint_{\re^N}\dint_{\re^N}
\dfrac{|w_n(x)-w_n(y)|^p}{|x-y|^{N+ps}|x|^\beta|y|^\beta}dxdy}{\dyle\int_{\ren}\dfrac{|w_n(x)|^p}{|x|^{ps+\beta}}dx}\\
&\le &
2\L(\g_0)\Big(1+\dfrac{I_n+J_n}{\dyle\int_{\ren}\dfrac{|w_n(x)|^p}{|x|^{ps+\beta}}dx}\Big).
\end{eqnarray}
Since
$\dyle\int_{\ren}\dfrac{|w_n(x)|^p}{|x|^{ps+\beta}}dx\uparrow
\infty$ as $n\to \infty$, then passing to the limit in
\eqref{KPL}, it follows that
$$
\L_{N,p,s,\g}\le 2\L(\g_0)
$$
and then the result follows.
\end{proof}
\

In the sequel we need to use a version of the Hardy inequality in
bounded domains. More precisely, we have the next result.
\begin{Lemma}\label{cor}
Let $\Omega$ be a bounded regular domain such that $0\in \Omega$,
then there exists a constant $C\equiv C(\Omega,s,p,N)>0$ such that
for all $u\in \mathcal{C}^\infty_0(\Omega)$, we have
\begin{equation}\label{boun}
\begin{array}{rcl}
C\dint\limits_{\Omega} \dfrac{|u(x)|^p}{|x|^{ps+2\beta}}\,dx
&\leq&\dint\limits_{\Omega}\dint\limits_{\Omega}
\dfrac{|u(x)-u(y)|^p}{|x-y|^{N+ps}}\dfrac{dx}{|x|^{\beta}}
\dfrac{dy}{|y|^\beta}.
\end{array}
\end{equation}
\end{Lemma}
\begin{proof} Fix  $u\in \mathcal{C}^\infty_0(\Omega)$ and let $\tilde{u}$, be the extension of $u$ to $\ren$ defined in Lemma \ref{ext}.
Then from Theorem \ref{HHH}, we get
$$
2\L(\g)\dint\limits_{\re^N}
\dfrac{|\tilde{u}(x)|^p}{|x|^{ps+2\beta}}\,dx
\leq\dint\limits_{\re^N}\dint\limits_{\re^N}
\dfrac{|\tilde{u}(x)-\tilde{u}(y)|^p}{|x-y|^{N+ps}}\dfrac{dx}{|x|^{\beta}}
\dfrac{dy}{|y|^\beta}\le
||\tilde{u}||^p_{{X^{s,p,\beta}(\ren)}}\le
C||u||^p_{{X^{s,p,\beta}(\Omega)}}.$$ Since
$\tilde{u}_{|\Omega}=u$, form Remark \ref{equiv} we conclude
that
\begin{equation*}
\begin{array}{lll}
\dyle 2\L(\g)\dint\limits_{\Omega}
\dfrac{|u(x)|^p}{|x|^{ps+2\beta}}\,dx  & \leq &
C||u||^p_{{X^{s,p,\beta}(\Omega)}}\\
&\le & C_1\dyle
|||u|||^p_{{X^{s,p,\beta}_0(\Omega)}}=C_1\dint\limits_{\Omega}\dint\limits_{\Omega}
\dfrac{|u(x)-u(y)|^p}{|x-y|^{N+ps}}\dfrac{dx}{|x|^{\beta}}
\dfrac{dy}{|y|^\beta}.\end{array} \end{equation*} Hence we reach
the desired result.
\end{proof}

Now, we are able to proof Theorem \ref{main}.

{\bf Proof of Theorem \ref{main}. }

We follow closely the arguments used in \cite{APP}. Let $u \in
\mathcal{C}_0^\infty(\Omega)$ and define $\alpha=\dfrac{N-ps}{p}$,
then $w(x)=|x|^{-\alpha}$ and $v(x)= \dfrac{u(x)}{w(x)}$.

Recall that from the result of \cite{FS}, we have
\begin{equation}\label{FS1} h_s(u)\geq
C\dint_{\ren}\dint_{\ren}\dfrac{|v(x)-v(y)|^p}{|x-y|^{N+ps}}\dfrac{\,dx}{|x|^{\frac{N-ps}{2}}}\dfrac{\,dy}{|y|^{\frac{N-ps}{2}}}.
\end{equation}
Let us analyze the right hand side of the previous
inequality.

Notice that
$$
\begin{array}{rcl}
 \dfrac{|v(x)-v(y)|^p}{|x-y|^{N+ps}}\, w(x)^{\frac{p}{2}} w(y)^{\frac{p}{2}} &=& \dfrac{|w(y)u(x)-w(x)u(y)|^p}{|x-y|^{N+ps}}\dfrac{1}{(w(x)w(y))^{\frac{p}{2}}}\\
 \\
 \\
 &=& \dfrac{\big|(u(x)-u(y))-\dfrac{u(y)}{w(y)}(w(x)-w(y))\big|^p}{|x-y|^{N+ps}}\left(\dfrac{w(y)}{w(x)}\right)^{\frac{p}{2}}= f_{1}(x,y).
\end{array}
$$
In the same way, thanks to the symmetry of $f_1(x,y)$, it immediately follows that
$$
\dfrac{|v(x)-v(y)|^p}{|x-y|^{N+ps}}\, (w(x))^{\frac{p}{2}} (w(y))^{\frac{p}{2}} = \dfrac{\big|(u(y)-u(x))-\dfrac{u(x)}{w(x)}(w(y)-w(x))\big|^p}{|x-y|^{N+ps}}\left(\dfrac{w(x)}{w(y)}\right)^{\frac{p}{2}}=f_2(x,y).
$$
Hence,
$$
h_s(u)\ge \dfrac{1}{2}\dint_{\re^N}\dint_{\re^N}
f_{1}(x,y)\,dx\,dy + \dfrac{1}{2}\dint_{\re^N}\dint_{\re^N}
f_{2}(x,y)\,dx\,dy.
$$
Since $f_1$ and $f_2$ are positive functions, it follows that
$$
h_s(u)\geq \dfrac{1}{2}\dint_{\Omega}\dint_{\Omega} f_{1}(x,y)\,dx\,dy + \dfrac{1}{2}\dint_{\Omega}\dint_{\Omega} f_{2}(x,y)\,dx\,dy.
$$
Using the fact that $\O$ is a bounded domain, we obtain that for
all $(x,y)\in (\Omega\times\Omega)$ and $q<p$,
$$
\dfrac{1}{|x-y|^{N+ps}} \geq \dfrac{C(\Omega)}{|x-y|^{N+qs}}
$$
and
$$
Q(x,y)\equiv
\frac{\left(w(x)w(y)\right)^{\frac{p}{2}}}{w(x)^p+w(y)^p}\le C.
$$
Define
$$
D(x,y)\equiv\left(\frac{w(x)}{w(y)}\right)^{\frac{p}{2}}+\left(\frac{w(y)}{w(x)}\right)^{\frac{p}{2}}\equiv\frac{w(x)^p+w(y)^p}{\left(w(x)w(y)\right)^{\frac{p}{2}}},
$$
then $Q(x,y)D(x,y)=1$. Thus
\begin{eqnarray*}
&f_{1}(x,y)\geq C(\Omega)
Q(x,y)\left(\dfrac{w(y)}{w(x)}\right)^{\frac{p}{2}} \times\\
& \Big[ \dfrac{|u(x)-u(y)|^p}{|x-y|^{N+qs}} -p
\dfrac{|u(x)-u(y)|^{p-2}}{|x-y|^{N+qs}}\big\langle
u(x)-u(y),\dfrac{u(y)}{w(y)}(w(x)-w(y))\big\rangle \\
&
+C(p)\dfrac{|\dfrac{u(y)}{w(y)}(w(x)-w(y))|^p}{|x-y|^{N+qs}}\Big].
\end{eqnarray*}
Hence
$$
\begin{array}{rcl}
f_{1}(x,y)&\geq& \Big[C(\Omega)
Q(x,y)\left(\dfrac{w(y)}{w(x)}\right)^{\frac{p}{2}}
 \dfrac{|u(x)-u(y)|^p}{|x-y|^{N+qs}}\Big]\\
 &&\\
     &-&\Big[p C(\Omega) Q(x,y)\left(\dfrac{w(y)}{w(x)}\right)^{\frac{p}{2}} \dfrac{|u(x)-u(y)|^{p-1}}{|x-y|^{N+qs}}\big|\dfrac{u(y)}{w(y)}\big||(w(x)-w(y))| \Big].
     \end{array}
$$
In the same way we reach that
$$
\begin{array}{rcl}
f_{2}(x,y)&\geq& \Big[C(\Omega)
Q(x,y)\left(\dfrac{w(x)}{w(y)}\right)^{\frac{p}{2}}
 \dfrac{|u(y)-u(x)|^p}{|x-y|^{N+qs}}\Big]\\
 &&\\
 &-&\Big[p C(\Omega) Q(x,y)\left(\dfrac{w(x)}{w(y)}\right)^{\frac{p}{2}} \dfrac{|u(x)-u(y)|^{p-1}}{|x-y|^{N+qs}}\big|\dfrac{u(x)}{w(x)}\big||(w(x)-w(y))| \Big].
 \end{array}
$$
Therefore,
$$
\begin{array}{rcl}
h_s(u)&\geq& C(\Omega) \dint_{\Omega}\dint_{\Omega}Q(x,y)\Big(\left(\dfrac{w(y)}{w(x)}\right)^{\frac{p}{2}}+\left(\dfrac{w(x)}{w(y)}\right)^{\frac{p}{2}} \Big)\dfrac{|u(x)-u(y)|^p}{|x-y|^{N+qs}}\,dx\,dy \\
&&\\
&-&pC(\Omega)\dint_{\Omega}\dint_{\Omega} \Big[ Q(x,y)\left(\dfrac{w(y)}{w(x)}\right)^{\frac{p}{2}} \dfrac{|u(x)-u(y)|^{p-1}}{|x-y|^{N+qs}}\big|\dfrac{u(y)}{w(y)}\big||(w(x)-w(y))| \Big]\,dx\,dy\\
&&\\
&-&pC(\Omega)\dint_{\Omega}\dint_{\Omega} \Big[
Q(x,y)\left(\dfrac{w(x)}{w(y)}\right)^{\frac{p}{2}}
\dfrac{|u(x)-u(y)|^{p-1}}{|x-y|^{N+qs}}\big|\dfrac{u(x)}{w(x)}\big||(w(x)-w(y))|
\Big]\,dx\,dy.
\end{array}
$$
Thus
\begin{equation}
\begin{array}{rcl}
h_s(u)&\geq&
C(\Omega)\dint_{\Omega}\dint_{\Omega}\dfrac{|u(x)-u(y)|^p}{|x-y|^{N+qs}}\,dx\,dy\\
\\&-& C_1(\Omega,p)\dint_{\Omega}\dint_{\Omega}\big(h_{1}(x,y)+
h_{2}(x,y)\big)\,dx\,dy,
\end{array}
\end{equation} \label{hsu}
with
$$
h_{1}(x,y)= Q(x,y)\left(\dfrac{w(y)}{w(x)}\right)^{\frac{p}{2}}
\dfrac{|u(x)-u(y)|^{p-1}}{|x-y|^{N+qs}}\big|\dfrac{u(y)}{w(y)}\big||(w(x)-w(y))|,
$$
$$
h_{2}(x,y)= Q(x,y)\left(\dfrac{w(x)}{w(y)}\right)^{\frac{p}{2}}
\dfrac{|u(x)-u(y)|^{p-1}}{|x-y|^{N+qs}}\big|\dfrac{u(x)}{w(x)}\big||(w(x)-w(y))|.
$$
Since $h_{1}(x,y)$ and $h_{2}(x,y)$ are symmetric functions, we
just have to estimate $\dint_{\Omega}\dint_{\Omega}
h_{2}(x,y)\,dx\,dy.$

Using Young inequality, we get
\begin{equation}
\begin{array}{rcl}
\dyle \dint_{\Omega}\dint_{\Omega} h_{2}(x,y) dxdy &\le & \dyle
\e\dint_{\Omega}\dint_{\Omega}\dfrac{|u(x)-u(y)|^p}{|x-y|^{N+qs}}\,dx\,dy\\
& + & \dyle C(\e)\dint_{\Omega}\dint_{\Omega}G(x,y)\,dx\,dy,
\end{array}
\end{equation} \label{h1h2}
with
$$
G(x,y)=(Q(x,y))^p\left(\dfrac{w(x)}{w(y)}\right)^{\frac{p^2}{2}}
\big|\dfrac{u(x)}{w(x)}\big|^p\frac{|w(x)-w(y)|^p}{|x-y|^{N+qs}}.
$$
We claim that
$$
I\equiv\dint_{\Omega}\dint_{\Omega}G(x,y)\,dx\,dy\le C
\dint\limits_{\re^N}\dint\limits_{\re^N}\dfrac{|v(x)-v(y)|^p}{|x-y|^{N+ps}}\dfrac{\,dx}{|x|^{\frac{N-ps}{2}}}\dfrac{\,dy}{|y|^{\frac{N-ps}{2}}}.
$$
Notice that
$$
I= \dint_{\Omega}\dint_{\Omega} \dfrac{(u(x))^p}{|x-y|^{N+qs}}
\dfrac{(w(x))^{p^{2}-p}|w(x)-w(y)|^p}{(w(x)^p+w(y)^p)^p}
\,dx\,dy,
$$
then
$$
I= \dint_{\Omega} u^p(x)\Big[\dint_{\Omega}
\dfrac{||x|^{\alpha}-|y|^{\alpha}|^p}{(|x|^{\alpha p} +
|y|^{\alpha p})^p} \dfrac{|y|^{\alpha p(p-1)}}{|x-y|^{N+qs}}
\,dy\Big]\,dx.
$$
To compute the above integral, we closely follow the arguments
used in \cite{FV}. We set $y=\rho y'$ and $x=rx'$ with
$|x'|=|y'|=1$, then taking in consideration that $\O\subset
B_0(R)$, it follows that
\begin{eqnarray*}
I & = & \dint_{\Omega} u^p(x)\Big[\dint_{\Omega} \dfrac{||x|^{\alpha}-|y|^{\alpha}|^p}{(|x|^{\alpha p} + |x|^{\alpha p})^p} \dfrac{|y|^{\alpha p(p-1)}}{|x-y|^{N+qs}} \,dy\Big]\,dx \\
& \le & \dint_{\Omega} u^p(x)\int_0^R
\dfrac{(|r^{\alpha}-\rho^{\alpha}|^p \rho^{\alpha
p(p-1)+N-1}}{(r^{p\alpha}+\rho^{p\alpha})^p}\Big(\dint_{{\mathbb{S}}^{N-1}}\dfrac{dy'}{|\rho
y'-rx'|^{N+qs}}\Big)d\rho dx.
\end{eqnarray*}
We set $\rho=r\sigma$, then
\begin{eqnarray*}
I & \le & \dint_{\Omega}
\frac{u^p(x)}{|x|^{qs}}\int_0^{\frac{R}{r}}
\dfrac{|1-\sigma^{\alpha}|^p\sigma^{\alpha
p(p-1)+N-1}}{(1+\sigma^{\alpha
p})^p}\Big(\dint_{{\mathbb{S}}^{N-1}}\dfrac{dy'}{|\sigma
y'-x'|^{N+qs}}\Big)d\sigma dx\\
&=& \dint_{\Omega} \frac{u^p(x)}{|x|^{qs}}\int_0^{\frac{R}{r}}
\dfrac{|1-\sigma^{\alpha}|^p\sigma^{\alpha
p(p-1)+N-1}}{(1+\sigma^{\alpha p})^p}K(\sigma)d\sigma dx\le
\mu\dint_{\Omega} \frac{u^p(x)}{|x|^{qs}}dx,
\end{eqnarray*}
where $$ \mu=\int_0^{\infty}
\dfrac{|1-\sigma^{\alpha}|^p\sigma^{\alpha
p(p-1)+N-1}}{(1+\sigma^{\alpha p})^p}K(\sigma)d\sigma$$ and
$$
K(\sigma)=2\dfrac{\pi^{\frac{N-1}{2}}}{\Gamma(\frac{N-1}{2})}\int_0^\pi
\dfrac{\sin^{N-2}(\theta)}{(1-2\sigma \cos
(\theta)+\sigma^2)^{\frac{N+qs}{2}}}d\theta.
$$
Let us show that $\mu<\infty$.

It is clear that, as $\sigma \to \infty$, we have
$$
\dfrac{(|1-\sigma^{\alpha}|^p\sigma^{\alpha
p(p-1)+N-1}}{(1+\sigma^{\alpha p})^p}K(\sigma)\backsimeq
\sigma^{-1-qs}\in L^1(1,\infty).
$$
Now, taking in consideration that $K(\s)\le C|1-\s|^{-1-ps}$ as
$s\to 1$, and following the same computation as in Lemma
\ref{ictp1}, it follows that $$ \int_0^{1}
\dfrac{(1-\sigma^{\alpha})^p\sigma^{\alpha
p(p-1)+N-1}}{(1+\sigma^{\alpha p})^p}K(\sigma)d\sigma<\infty.
$$
Thus $\mu<\infty$.

Hence combining the above estimates, there results that
$$
I\le C\dint_{\Omega} \frac{u^p(x)}{|x|^{qs}}dx.
$$
Since $u(x)=v(x)|x|^{-(\frac{N-ps}{p})}$, then
$$
I\le C\dint_{\Omega} \dfrac{|v(x)|^p}{|x|^{N-s(p-q)}}\,dx.
$$
Let $\beta_0=\frac{N-ps}{2}+\frac{(q-p)s}{2}$, then
$\beta_0<\frac{N-ps}{2}$. Applying Lemma \ref{cor}, we obtain that
$$
\begin{array}{rcl}
I&\le &C(\Omega)\dint\limits_{\Omega}\dint\limits_{\Omega}\dfrac{|v(x)-v(y)|^p}{|x-y|^{N+ps}|x|^{\beta_0}|y|^{\beta_0}} \,dy\,dx \\
&\le & C_1(\Omega)\dint\limits_{\Omega}\dint\limits_{\Omega}\dfrac{|v(x)-v(y)|^p}{|x-y|^{N+ps}|x|^{\frac{N-ps}{2}}|y|^{\frac{N-ps}{2}}} \,dy\,dx \\
&\le & C_1(\Omega)\dint\limits_{\re^N}\dint\limits_{\re^N}
\dfrac{|v(x)-v(y)|^p}{|x-y|^{N+ps}|x|^{\frac{N-ps}{2}}|y|^{\frac{N-ps}{2}}}
\,dy\,dx.
\end{array}
$$
Therefore, using again estimate \eqref{FS1}, we reach that
$$
I\le C_2(\Omega)
\dint\limits_{\re^N}\dint\limits_{\re^N}\dfrac{|v(x)-v(y)|^p}{|x-y|^{N+ps}}\dfrac{\,dx}{|x|^{\frac{N-ps}{2}}}\dfrac{\,dy}{|y|^{\frac{N-ps}{2}}}
$$
and the claim follows.

As a direct consequence of the above estimates, we have proved
that
\begin{equation}\label{mad1}
\dint_{\Omega}\dint_{\Omega}\dfrac{|u(x)-u(y)|^p}{|x-y|^{N+qs}}\,dx\,dy\le
C_3\dint\limits_{\re^N}\dint\limits_{\re^N}\dfrac{|v(x)-v(y)|^p}{|x-y|^{N+ps}}\dfrac{\,dx}{|x|^{\frac{N-ps}{2}}}\dfrac{\,dy}{|y|^{\frac{N-ps}{2}}}.
\end{equation}
Thus
$$\dint_{\Omega}\dint_{\Omega}\dfrac{|u(x)-u(y)|^p}{|x-y|^{N+qs}}\,dx\,dy\le
C h_s(u),
$$
and the result follows at once. \cqd

We are now in position to prove the Theorem \ref{main01}.

{\bf Proof of Theorem \ref{main01}.} Recall that
$\a=\frac{N-ps}{p}$. Since $\a p^*_{s,q}=\frac{N(N-ps)}{N-qs}<N$,
it follows that $\dint\limits_{\O}
\dfrac{|u(x)|^{p^*_{s,q}}}{|x|^{\a p^*_{s,q}}}\,dx<\infty$, for
all $u \in \mathcal{C}_0^\infty(\ren)$.

To prove \eqref{sara00}, we will use estimate \eqref{mad1} and the
fractional Sobolev inequality.

Fix $u \in \mathcal{C}_0^\infty(\Omega)$ and define
$u_1(x)=\dfrac{u(x)}{|x|^{\a}}$. By \eqref{mad1}, we obtain that
\begin{equation*}
C(\O)\dint_{\Omega}\dint_{\Omega}\dfrac{|u_1(x)-u_1(y)|^p}{|x-y|^{N+qs}}\,dx\,dy\le
\dint\limits_{\re^N}\dint\limits_{\re^N}\dfrac{|u(x)-u(y)|^p}{|x-y|^{N+ps}}\dfrac{\,dx}{|x|^{\frac{N-ps}{2}}}\dfrac{\,dy}{|y|^{\frac{N-ps}{2}}}.
\end{equation*}
Now, using Sobolev inequality, there results that
$$
S\big(\io |u_1(x)|^{p^*_{s,q}}dx \Big)^{\frac{p}{p^*_{s,q}}}\le
\dint_{\Omega}\dint_{\Omega}\dfrac{|u_1(x)-u_1(y)|^p}{|x-y|^{N+qs}}\,dx\,dy,
$$
where $p^*_{s,q}=\frac{pN}{N-qs}$. Hence, substituting $u_1$ by
its value, we get
\begin{equation}\label{sara11}
\Big(\dint\limits_{\O} \dfrac{|u(x)|^{p^*_{s,q}}}{|x|^{\a
p^*_{s,q}}}\,dx\Big)^{\frac{p}{p^*_{s,q}}}\le
C\dint_{\ren}\dint_{\ren}\dfrac{|u(x)-u(y)|^p}{|x-y|^{N+ps}}\dfrac{\,dx}{|x|^{\b}}\dfrac{\,dy}{|y|^{\b}}.
\end{equation}
If we set $\beta=\frac{N-ps}{2}=\alpha \frac{p}{2}$, then inequality \eqref{sara11}
can be written in the form
\begin{equation}\label{sara1100}
\Big(\dint\limits_{\O} \dfrac{|u(x)|^{p^*_{s,q}}}{|x|^{2\b
\frac{p^*_{s,q}}{p}}}\,dx\Big)^{\frac{p}{p^*_{s,q}}}\le
C\dint_{\ren}\dint_{\ren}\dfrac{|u(x)-u(y)|^p}{|x-y|^{N+ps}}\dfrac{\,dx}{|x|^{\b}}\dfrac{\,dy}{|y|^{\b}}.
\end{equation}
 \cqd

\

As a consequence, we will prove the fractional
Caffarelli-Kohn-Nirenberg inequality given in Theorem \ref{CKN01}.

{\bf Proof of Theorem \ref{CKN01}.} Let  $u \in
\mathcal{C}_0^\infty(\ren)$, without loss of generality, we can
assume that $u\ge 0$. Using the fact that $\beta<\frac{N-ps}{2}$,
we easily get that $\dint\limits_{\re^N}
\dfrac{|u(x)|^{p^*_s}}{|x|^{2\beta\frac{p^*_s}{p}}}\,dx<\infty$.

From now and for simplicity of typing, we denote by $C,
C_1,C_2,...$ any universal constant that does not depend on $u$
and can change from a line to another.

We set $\tilde{u}(x)=\dfrac{u(x)}{w_1(x)}$, where
$w_1(x)=|x|^{\frac{2\beta}{p}}$, then
\begin{equation}\label{new1}
\Big(\dint\limits_{\re^N}
\dfrac{|u(x)|^{p^*_s}}{|x|^{2\beta\frac{p^*_s}{p}}}\,dx\Big)^{\frac{p}{p^*_s}}=\Big(\dint\limits_{\re^N}
|\tilde{u}|^{p^*_s}\,dx\Big)^{\frac{p}{p^*_s}}.
\end{equation}
Using Sobolev inequality, it follows that
\begin{equation}\label{new2}
S\Big(\dint\limits_{\re^N}
|\tilde{u}|^{p^*_s}\,dx\Big)^{\frac{p}{p^*_s}}\le
\dint\limits_{\re^N}\dint\limits_{\re^N}
\dfrac{|\tilde{u}(x)-\tilde{u}(y)|^p}{|x-y|^{N+ps}}dxdy.\end{equation}
To get the desired result we just have to show that

\begin{equation}\label{mainppp}
\dint\limits_{\re^N}\dint\limits_{\re^N}
\dfrac{|\tilde{u}(x)-\tilde{u}(y)|^p}{|x-y|^{N+ps}}dxdy\le C
\dint\limits_{\re^N}\dint\limits_{\re^N}
\dfrac{|u(x)-u(y)|^p}{|x-y|^{N+ps}}\dfrac{dx}{|x|^{\beta}}
\dfrac{dy}{|y|^\beta}\end{equation} for some positive constant
$C$.

Using the definition of $\tilde{u}$, we get
$$
\dint\limits_{\re^N}\dint\limits_{\re^N}
\dfrac{|u(x)-u(y)|^p}{|x-y|^{N+ps}}\dfrac{dx}{|x|^{\beta}}
\dfrac{dy}{|y|^\beta}=\dint\limits_{\re^N}\dint\limits_{\re^N}
\dfrac{|w_1(x)\tilde{u}(x)-w_1(y)\tilde{u}(y)|^p}{|x-y|^{N+ps}}\dfrac{dx}{w^{\frac{p}{2}}_1(x)}
\dfrac{dy}{w^{\frac{p}{2}}_1(y)}.
$$
Notice that
\begin{eqnarray*}
&\dfrac{|w_1(x)\tilde{u}(x)-w_1(y)\tilde{u}(y)|^p}{|x-y|^{N+ps}}\dfrac{1}{w^{\frac{p}{2}}_1(x)}
\dfrac{1}{w^{\frac{p}{2}}_1(y)}=\\ &
\dfrac{\big|(\tilde{u}(x)-\tilde{u}(y))-w_1(y)\tilde{u}(y)(\dfrac{1}{w_1(x)}-\dfrac{1}{w_1(y)})\big|^p}{|x-y|^{N+ps}}
\left(\dfrac{w_1(x)}{w_1(y)}\right)^{\frac{p}{2}}\equiv
\tilde{f}_{1}(x,y).
\end{eqnarray*}
In the same way we have
\begin{eqnarray*}
&\dfrac{|w_1(x)\tilde{u}(x)-w_1(y)\tilde{u}(y)|^p}{|x-y|^{N+ps}}\dfrac{1}{w^{\frac{p}{2}}_1(x)}
\dfrac{1}{w^{\frac{p}{2}}_1(y)}=\\
&
\dfrac{\big|(\tilde{u}(y)-\tilde{u}(x))-w_1(x)\tilde{u}(x)(\dfrac{1}{w_1(y)}-\dfrac{1}{w_1(x)})\big|^p}{|x-y|^{N+ps}}\left(\dfrac{w_1(y)}{w_1(x)}\right)^{\frac{p}{2}}\equiv
\tilde{f}_{2}(x,y).
\end{eqnarray*}
Since
$$
\dint\limits_{\re^N}\dint\limits_{\re^N} \tilde{f}_1(x,y)dxdy=
\dint\limits_{\re^N}\dint\limits_{\re^N} \tilde{f}_2(x,y)dxdy,
$$
we get
\begin{eqnarray}\label{somme}
&\dint\limits_{\re^N}\dint\limits_{\re^N}
\dfrac{|u(x)-u(y)|^p}{|x-y|^{N+ps}}\dfrac{dx}{|x|^{\beta}}
\dfrac{dy}{|y|^\beta}=\dfrac
12\dint\limits_{\re^N}\dint\limits_{\re^N} \tilde{f}_1(x,y)dxdy+
\frac 12\dint\limits_{\re^N}\dint\limits_{\re^N}
\tilde{f}_2(x,y)dxdy.
\end{eqnarray}

As in the proof of Theorem \ref{main}, we define
{
$$
Q_1(x,y)\equiv
\frac{\left(w_1(x)w_1(y)\right)^{\frac{p}{2}}}{w_1(x)^p+w_1(y)^p}\le C.
$$
It is clear that
$$
Q_1(x,y)\times \bigg( \left(\frac{w(x)}{w(y)}\right)^{\frac{p}{2}}+\left(\frac{w(y)}{w(x)}\right)^{\frac{p}{2}}\bigg)=1.
$$
}
Thus
\begin{eqnarray*}
&\tilde{f}_{1}(x,y)\geq { C Q_1(x,y)}
\left(\dfrac{w_1(x)}{w_1(y)}\right)^{\frac{p}{2}} \times\\
& \Big[ \dfrac{|\tilde{u}(x)-\tilde{u}(y)|^p}{|x-y|^{N+ps}}-p
\dfrac{|\tilde{u}(x)-\tilde{u}(y)|^{p-2}}{|x-y|^{N+ps}}\big\langle
\tilde{u}(x)-\tilde{u}(y),
w_1(y)\tilde{u}(y)(\dfrac{1}{w_1(x)}-\dfrac{1}{w_1(y)})\big\rangle\\
&+C(p)\dfrac{|
w_1(y)\tilde{u}(y)(\dfrac{1}{w_1(x)}-\dfrac{1}{w_1(y)})
|^p}{|x-y|^{N+ps}}\Big],
\end{eqnarray*}
Hence
\begin{eqnarray*}
&\tilde{f}_{1}(x,y)\geq
{ C Q_1(x,y)} \left(\dfrac{w_1(x)}{w_1(y)}\right)^{\frac{p}{2}} \times\\
& \Big[ \dfrac{|\tilde{u}(x)-\tilde{u}(y)|^p}{|x-y|^{N+ps}}-p
\dfrac{|\tilde{u}(x)-\tilde{u}(y)|^{p-1}}{|x-y|^{N+ps}}|
w_1(y)\tilde{u}(y)(\dfrac{1}{w_1(x)}-\dfrac{1}{w_1(y)})|\Big].
\end{eqnarray*}
Using Young inequality, we get the existence of $C_1,C_2>0$ such
that
\begin{eqnarray*}
&\tilde{f}_{1}(x,y)\geq { C Q_1(x,y)}\left(\dfrac{w_1(x)}{w_1(y)}\right)^{\frac{p}{2}} \times\\
& \Big[ C_1\dfrac{|\tilde{u}(x)-\tilde{u}(y)|^p}{|x-y|^{N+ps}}
-C_2\dfrac{|w_1(y)\tilde{u}(y)(\dfrac{1}{w_1(x)}-\dfrac{1}{w_1(y)})|^{p}}{|x-y|^{N+ps}}|\Big].
\end{eqnarray*}
In the same way and using that $\tilde{f}_1, \tilde{f}_2$ are
symmetric functions, it holds

\begin{eqnarray*}
&f_{2}(x,y)\geq { C Q_1(x,y)}\left(\dfrac{w_1(y)}{w_1(x)}\right)^{\frac{p}{2}} \times\\
& \Big[ C_1\dfrac{|\tilde{u}(x)-\tilde{u}(y)|^p}{|x-y|^{N+ps}}
-C_2\dfrac{|w_1(x)\tilde{u}(x)(\dfrac{1}{w_1(y)}-\dfrac{1}{w_1(x)})|^{p}}{|x-y|^{N+ps}}|\Big].
\end{eqnarray*}

Thus, form \eqref{somme}, we get the existence of positive constants $C_1, C_2, {C_3}$
such that
\begin{eqnarray*}
&\dint\limits_{\re^N}\dint\limits_{\re^N}
\dfrac{|u(x)-u(y)|^p}{|x-y|^{N+ps}}\dfrac{dx}{|x|^{\beta}}
\dfrac{dy}{|y|^\beta}\ge \\
&{C_1}\dint\limits_{\re^N}\dint\limits_{\re^N}
\dfrac{|\tilde{u}(x)-\tilde{u}(y)|^p}{|x-y|^{N+ps}}{Q_1(x,y)}\Big[\left(\dfrac{w_1(y)}{w_1(x)}\right)^{\frac{p}{2}}
+\left(\dfrac{w_1(x)}{w_1(y)}\right)^{\frac{p}{2}}\Big]dxdy\\
&-{C_2}\dint\limits_{\re^N}\dint\limits_{\re^N}{ Q_1(x,y)}\left(\dfrac{w_1(x)}{w_1(y)}\right)^{\frac{p}{2}}\dfrac{|w_1(y)\tilde{u}(y)(\dfrac{1}{w_1(x)}-\dfrac{1}{w_1(y)})|^{p}}{|x-y|^{N+qs}}|dxdy\\
&-{C_3}\dint\limits_{\re^N}\dint\limits_{\re^N} { Q_1(x,y)}\left(\dfrac{w_1(y)}{w_1(x)}\right)^{\frac{p}{2}}\dfrac{|w_1(x)\tilde{u}(x)(\dfrac{1}{w_1(y)}-\dfrac{1}{w_1(x)})|^{p}}{|x-y|^{N+qs}}dxdy.
\end{eqnarray*}
Since
{$$
Q_1(x,y)\Big[\left(\dfrac{w_1(y)}{w_1(x)}\right)^{\frac{p}{2}}
+\left(\dfrac{w_1(x)}{w_1(y)}\right)^{\frac{p}{2}}\Big]=1,
$$
}
then
\begin{equation}\label{new4}
\begin{array}{lll}
&\dint\limits_{\re^N}\dint\limits_{\re^N}
\dfrac{|\tilde{u}(x)-\tilde{u}(y)|^p}{|x-y|^{N+ps}}dxdy\le
C_1\dint\limits_{\re^N}\dint\limits_{\re^N}
\dfrac{|u(x)-u(y)|^p}{|x-y|^{N+ps}}\dfrac{dx}{|x|^{\beta}}
\dfrac{dy}{|y|^\beta}\\
&+C_2\dint\limits_{\re^N}\dint\limits_{\re^N}{ Q_1(x,y)}\left(\dfrac{w_1(x)}{w_1(y)}\right)^{\frac{p}{2}}\dfrac{|w_1(y)\tilde{u}(y)(\dfrac{1}{w_1(x)}-\dfrac{1}{w_1(y)})|^{p}}{|x-y|^{N+qs}}|dxdy\\
&+C_3\dint\limits_{\re^N}\dint\limits_{\re^N}{ Q_1(x,y)}\left(\dfrac{w_1(y)}{w_1(x)}\right)^{\frac{p}{2}}\dfrac{|w_1(x)\tilde{u}(x)(\dfrac{1}{w_1(y)}-\dfrac{1}{w_1(x)})|^{p}}{|x-y|^{N+qs}}dxdy.
\end{array}
\end{equation}
We set
$$
g_1(x,y)={ Q_1(x,y)}\left(\dfrac{w_1(y)}{w_1(x)}\right)^{\frac{p}{2}}\dfrac{|w_1(x)\tilde{u}(x)(\dfrac{1}{w_1(y)}-\dfrac{1}{w_1(x)})|^{p}}{|x-y|^{N+ps}}
$$
and
$$
g_2(x,y)={ Q_1(x,y)}\left(\dfrac{w_1(x)}{w_1(y)}\right)^{\frac{p}{2}}\dfrac{|w_1(y)\tilde{u}(y)(\dfrac{1}{w_1(x)}-\dfrac{1}{w_1(y)})|^{p}}{|x-y|^{N+ps}}.
$$
It is clear that
$$
\dint\limits_{\re^N}\dint\limits_{\re^N} g_1(x,y)dxdy=
\dint\limits_{\re^N}\dint\limits_{\re^N} g_2(x,y)dxdy,
$$
therefore, to get the desired result, we just have to show that
$$
\dint\limits_{\re^N}\dint\limits_{\re^N} g_1(x,y)dxdy\le C
\dint\limits_{\re^N}\dint\limits_{\re^N}
\dfrac{|u(x)-u(y)|^p}{|x-y|^{N+ps}}\dfrac{dx}{|x|^{\beta}}
\dfrac{dy}{|y|^\beta}.
$$
Going back to the definition of $\tilde{u}$ and $w_1$, we reach
that
{$$
g_1(x,y)=\dfrac{|u(x)|^p\Big||x|^{\frac{2\beta}{p}}-|y|^{\frac{2\beta}{p}}\Big|^p}{|x|^{3\beta}|y|^{\beta}|x-y|^{N+ps}} \frac{|x|^\beta |y|^\beta}{|x|^{2\beta}+|y|^{2\beta}}
$$
}
We closely follow the same type of computation as in the proof of
Lemma \ref{ictp1}.

We have
{
\begin{eqnarray*}
\dint\limits_{\re^N}\dint\limits_{\re^N} g_1(x,y)dxdy &=&
\dint\limits_{\re^N}\dint\limits_{\re^N}\dfrac{|u(x)|^p\Big||x|^{\frac{2\beta}{p}}-|y|^{\frac{2\beta}{p}}\Big|^p}{|x|^{3\beta}|y|^{\beta}|x-y|^{N+ps}}  \frac{|x|^\beta |y|^\beta}{|x|^{2\beta}+|y|^{2\beta}} dxdy\\
&=&\dint\limits_{\re^N}\dfrac{|u(x)|^p}{|x|^{2\beta}}
\Big(\dint\limits_{\re^N}\dfrac{\Big||x|^{\frac{2\beta}{p}}-|y|^{\frac{2\beta}{p}}\Big|^p}{(|x|^{2\beta}+|y|^{2\beta})|x-y|^{N+ps}}dy\Big)dx.
\end{eqnarray*}
}
We set $r=|x|$ and $\rho=|y|$, then $x=rx', y=\rho y'$ with
$|x'|=|y'|=1$, then
\begin{eqnarray*}
&\dint\limits_{\re^N}\dfrac{|u(x)|^p}{|x|^{2\beta}}
\Big(\dint\limits_{\re^N}\dfrac{\Big||x|^{\frac{2\beta}{p}}-|y|^{\frac{2\beta}{p}}\Big|^p}{(|x|^{2\beta}+|y|^{2\beta})|x-y|^{N+ps}}dy\Big)dx=\\
&\dint\limits_{\re^N}\dfrac{|u(x)|^p}{|x|^{2\beta}} \Big[
\dint\limits_0^{+\infty}\dfrac{|r^{\frac{2\beta}{p}}-\rho^{\frac{2\beta}{p}}|^p\rho^{N-1}}{(r^{2\beta}+\rho^{2\beta})}\left(
\dint\limits_{|y'|=1}\dfrac{dH^{n-1}(y')}{|r x'-\rho y'|^{N+ps}}
\right) \,d\rho\Big]dx.
\end{eqnarray*}
Let $\s=\frac{\rho}{r}$, then
{
\begin{eqnarray*}
&\dint\limits_{\re^N}\dfrac{|u(x)|^p}{|x|^{2\beta}}
\Big(\dint\limits_{\re^N}\dfrac{\Big||x|^{\frac{2\beta}{p}}-|y|^{\frac{2\beta}{p}}\Big|^p}{(|x|^{2\beta}+|y|^{2\beta})|x-y|^{N+ps}}dy\Big)dx=
\dint\limits_{\re^N}\dfrac{|u(x)|^p}{|x|^{2\beta+ps}}
\Big[
\dint\limits_0^{+\infty}\frac{|1-\s^{\frac{2\beta}{p}}|^p}{1+\s^{2\beta}}\s^{N-1}K(\s)\,d\s\Big]dx,
\end{eqnarray*}
}
where $K$ is defined in \eqref{kkk}. Since
$$
\dint\limits_0^{+\infty}\frac{|1-\s^{\frac{2\beta}{p}}|^p}{1+\s^{2\beta}}\s^{N-1}K(\s)\,d\s\equiv
C_3<\infty,$$ it follows that
$$
\dint\limits_{\re^N}\dint\limits_{\re^N}
g_1(x,y)dxdy=C_3\dint\limits_{\re^N}
\dfrac{|u(x)|^p}{|x|^{2\beta+ps}}dx.
$$
Now, using inequality \eqref{IngL1}, we get
\begin{equation}\label{new3}
\dint\limits_{\re^N}\dint\limits_{\re^N} g_1(x,y)dxdy \leq C_4
\dint\limits_{\re^N}\dint\limits_{\re^N}
\dfrac{|u(x)-u(y)|^p}{|x-y|^{N+ps}}\dfrac{dx}{|x|^{\beta}}
\dfrac{dy}{|y|^\beta}.
\end{equation}
Combining \eqref{new1}, \eqref{new2}, \eqref{new3} and
\eqref{new4}, we reach the desired result. \cqd

In the case where $\O$ is a regular bounded domain containing the
origin, we have the following \emph{version} of Theorem
\ref{CKN01}.

\begin{Theorem}\label{last00}
Assume that $\O$ is a regular bounded domain with $0\in \O$, then
there exists a positive constant $C\equiv C(\O,N,p,s,\b)$ such
that for all $\phi\in \mathcal{C}^\infty_0(\Omega)$, we have
\begin{equation}\label{CKNNN00}
\begin{array}{rcl}
\dint_{\Omega}\dint_{\Omega}
\dfrac{|\phi(x)-\phi(y)|^p}{|x-y|^{N+ps}}\dfrac{dx}{|x|^{\beta}}
\dfrac{dy}{|y|^\beta}\ge C\Big(\dint\limits_{\Omega}
\dfrac{|\phi(x)|^{p^*_s}}{|x|^{2\beta\frac{p^*_s}{p}}}\,dx\Big)^{\frac{p}{p^*_s}}.
\end{array}
\end{equation}
\end{Theorem}

\begin{proof}
Let $\phi\in \mathcal{C}^\infty_0(\Omega)$ and define
$\tilde{\phi}$ to be the extension of $\phi$ to $\ren$ given in
Lemma \ref{ext}, then using the fact that $\O$ is a regular
bounded domain, we reach that
$$
||\tilde{\phi}||_{X^{s,p,\beta}(\ren)}\le C_1
||\phi||_{X^{s,p,\beta}(\Omega)}\le C_1
\Big(\dint_{\O}\dint_{\O}\dfrac{|\phi(x)-\phi(y)|^p}{|x-y|^{N+ps}}\dfrac{dxdy}{|x|^\beta|y|^\beta}\Big)^{\frac
1p}.
$$
Now, applying Theorem \ref{CKN01} to $\tilde{\phi}$, it follows
that
$$
\dint\limits_{\re^N}\dint\limits_{\re^N}
\dfrac{|\tilde{\phi}(x)-\tilde{\phi}(y)|^p}{|x-y|^{N+ps}}\dfrac{dx}{|x|^{\beta}}
\dfrac{dy}{|y|^\beta}\ge S(\b)\Big(\dint\limits_{\re^N}
\dfrac{|\tilde{\phi}(x)|^{p^*_s}}{|x|^{2\beta\frac{p^*_s}{p}}}\,dx\Big)^{\frac{p}{p^*_s}}.
$$
Hence combining the above estimates we get the desired result.
\end{proof}

\section{Application}\label{appl}
In this section we deal with the next problem
\begin{equation}\label{prob}
\left\{
\begin{array}{rcl}
L_{p,s}\, u&=&\l\dfrac{u^{p-1}}{|x|^{ps}}+u^{q}, \quad  u>0  \hbox{ in  } \Omega,\\
u&=&0\quad \hbox{  in  } \mathbb{R}^N\setminus\Omega,
\end{array}
\right.
\end{equation}
where $$ L_{s,p}\, u:=\mbox{ P.V. }\int_{\mathbb{R}^{N}} \,
\frac{|u(x)-u(y)|^{p-2}(u(x)-u(y))}{|x-y|^{N+ps}}\,dy, $$ and
$0<\l\le \L_{N,p,s}$.

In the case where $0<q<p-1$, the existence result follows using
variational arguments. More precisely we have:

\begin{enumerate}
\item If $\l<\L_{N,p,s}$, then the existence of a solution $u$ to
\eqref{prob} follows using classical minimizing argument. In this
case $u\in W^{s,p}_0(\O)$.

\item If $\l=\L_{N,p,s}$, the existence result follows using the
improved Hardy inequality in Theorem \ref{main}. In this case $u$
satisfies $h_{s,\O}(u)<\infty$ where $h_{s,\O}$ is defined by
\begin{equation}\label{TRT}
h_{s,\O}(u)\equiv \dint_{\re^N}\dint_{\re^N}
\dfrac{|u(x)-u(y)|^{p}}{|x-y|^{N+ps}}dx dy -\Lambda_{N,p,s}
\dint_{\O} \dfrac{|u(x)|^p}{|x|^{ps}} dx.
\end{equation}
This clearly implies that
$$
\dint_{\Omega}\dint_{\Omega}\dfrac{|u(x)-u(y)|^p}{|x-y|^{N+qs}}dxdy<\infty
\mbox{  for all }q<p.
$$
\end{enumerate}

We deal now with the case $q>p-1$.

Define $w(x)=|x|^{-\g}$ with $0<\g<\dfrac{N-ps}{p-1}$, then we
have previously obtained that
$$
L_{s,p}(w)= \L(\g)\dfrac{w^{p-1}}{|x|^{ps}}\:\:a.e. \mbox{ in
}\ren\backslash\{0\},
$$
where
$$
\L(\g)=\dint\limits_1^{+\infty}K(\sigma)(\sigma^\g-1)^{p-1}\left(\sigma^{N-1-\g(p-1)}-\sigma^{ps-1}\right)\,d\sigma,
$$
and $K$ is given by \eqref{kkk}. Let us begin by proving the next
lemma.
\begin{Lemma}\label{lm01}
Assume that $0<\l<\L_{N,p,s}$, then there exist $\g_1,\g_2$ such
that
$$
0<\g_1<\frac{N-ps}{p}<\g_2,
$$
and $\L(\g_1)=\L(\g_2)=\l$.
\end{Lemma}
\begin{proof}
We have $\L(0)=0, \L(\frac{N-ps}{p})=\L_{N,p,s}$, $\L(\g)<0$ if
$\g>\frac{N-ps}{p-1}$ and
$$
\L'(\g)=(p-1)\dint\limits_1^{+\infty}K(\sigma)\log(\s)(\sigma^\g-1)^{p-2}\left(\sigma^{N-1-\g(p-1)}-\sigma^{ps+\g-1}\right)\,d\sigma.
$$
It is clear that for $\g_0=\frac{N-ps}{p}$, we have $\L'(\g_0)=0$,
$\L'(\g)>0$ if $\g<\g_0$ and $\L'(\g)<0$ if $\g>\g_0$.

Hence, since $\l<\L_{N,p,s}$, we get the existence of
$0<\g_1<\frac{N-ps}{p}<\g_2<\frac{N-ps}{p-1}$ such that
$\L(\g_1)=\L(\g_2)=\l$.
\end{proof}
$$ \centering \includegraphics[width=7cm,height=50mm]{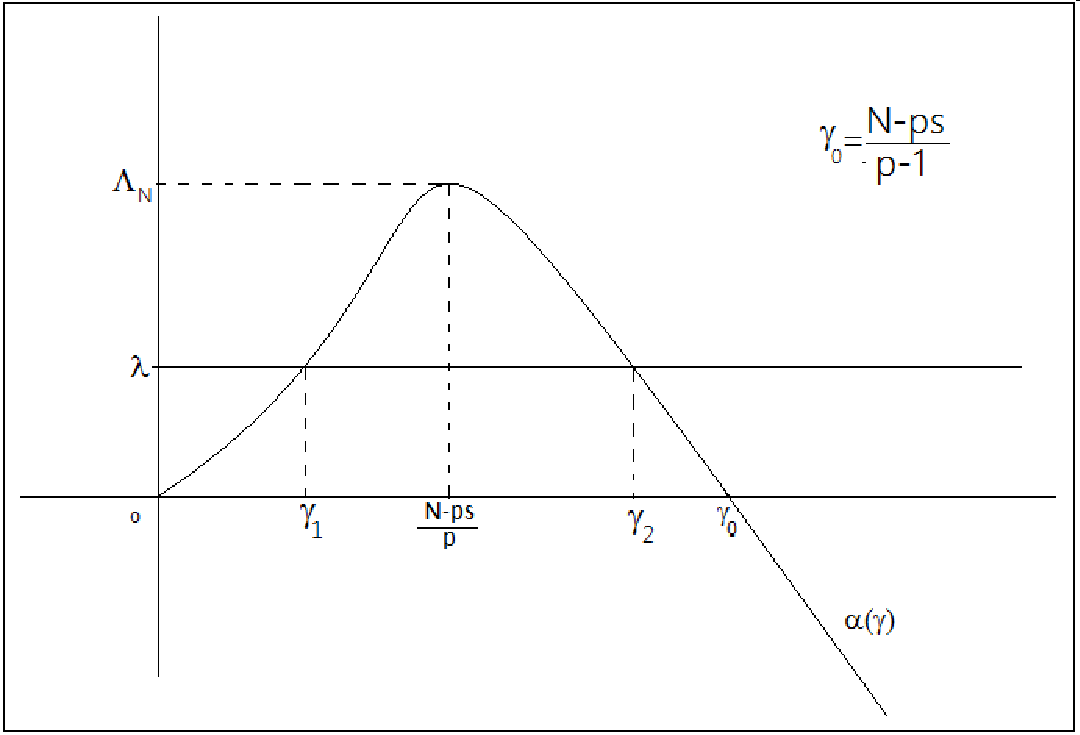}$$ \\

Define $q_+(p,s)=p-1+\frac{ps}{\g_1}$, it is clear that
$p^*_s-1<q_+(p,s)$. We have the next existence result.

\begin{Theorem}\label{th:exis}
Assume that $q<q_+(p,s)$, then
\begin{enumerate}
\item If $p-1<q<p^*_s-1$, problem \eqref{prob} has a solution $u$.
Moreover, $u\in W^{s,p}_0(\O)$ if $\l<\L_{N,p,s}$ and
$h_{s,\O}(u)<\infty$ if $\l=\L_{N,p,s}$ where $h_{s,\O}$ is
defined in \eqref{TRT}.

\item If $p^*_s-1\le q<q_+(p,s)$, then problem \eqref{prob} has a
positive supersolution $u$.
\end{enumerate}
\end{Theorem}
\begin{proof}
Let us begin with the case where $p-1<q<p^*_s-1$. If
$\l<\L_{N,p,s}$, then using the Mountain Pass Theorem, see
\cite{SV}, we get a positive solution $u\in W^{s,p}_0(\O)$.
However, if $\l=\L_{N,p,s}$, then using the improved Hardy
inequality in Theorem \ref{main} and the Mountain pass Theorem, we
reach a positive solution $u$ to problem \eqref{prob} with
$h_{s,\O}(u)<\infty$.

Assume now that $p^*_s-1\le q<q_+(p,s)$ and fix $\l_1\in
(\l,\L_{N,p,s})$ to be chosen later.

Let $\g_1\in (0,\frac{N-ps}{p})$ be such that $\G(\g_1)=\l_1$ and
set $w(x)=|x|^{-\g_1}$, then
$$
L_{s,p}(w)= \l_1\dfrac{w^{p-1}(x)}{|x|^{ps}}\:\:\: a.e \mbox{ in
}\ren\backslash\{0\}
$$
with  $\dfrac{w^{p-1}}{|x|^{ps}}\in L^1_{loc}(\ren)$. Hence
$$
L_{s,p}(w)=
\l\dfrac{w^{p-1}(x)}{|x|^{ps}}+(\l_1-\l)\dfrac{w^{p-1}(x)}{|x|^{ps}}\:\:\:
a.e \mbox{ in }\ren\backslash\{0\}.
$$
Using the fact that $q<q_+(p,s)$, we can choose $\l_1>\l$, very
close to $\l$ such that $\g_1(p-1)+ps>q\g_1$, thus, in any bounded
domain $\O$, we have
$$
(\l_1-\l)\dfrac{w^{p-1}(x)}{|x|^{ps}}\ge C(\O)w^q.
$$
Define $\hat{w}=Cw$, by the previous estimates, we can choose
$C(\O)>0$ such that $\hat{w}$ will be a supersolution to
\eqref{prob} in $\O$. Hence the result follows.

\end{proof}

Now, we show the optimality of the exponent $q_+(p,s)$. We have
the following non existence result.

\begin{Theorem}\label{th:non}
Let $q_+(p,s)=p-1+\frac{ps}{\g_1}$. If $q>q_+(p,s)$, then the
unique nonnegative supersolution $u\in W^{s,p}_{loc}(\O)$ to
problem \eqref{prob} is $u\equiv 0$.
\end{Theorem}

We first prove the next lemma which shows that the Hardy constant
is independent of the domain.

\begin{Lemma}\label{indep}
Let $\O\subset \ren$ be a regular domain such that $0\in \O$.
Define
$$
\L(\O)=\inf_{\{\phi\in \mathcal{C}^\infty_0(\O)\backslash
0\}}\dfrac{\dint_{\re^N}\dint_{\re^N}
\dfrac{|\phi(x)-\phi(y)|^p}{|x-y|^{N+ps}}dxdy}{\dyle\io\dfrac{|\phi(x)|^p}{|x|^{ps}}dx},
$$
then $\L(\O)=\L_{N,p,s}$ defined in \eqref{LL}.
\end{Lemma}
\begin{proof}
Recall that
$$
\L_{N,p,s}=\inf_{\{\phi\in \mathcal{C}^\infty_0(\ren)\backslash
0\}}\dfrac{\dint_{\re^N}\dint_{\re^N}
\dfrac{|\phi(x)-\phi(y)|^p}{|x-y|^{N+ps}}dxdy}{\dyle\irn\dfrac{|\phi(x)|^p}{|x|^{ps}}dx},$$
thus $\L(\O)\ge \L_{N,p,s}$. It is clear that if $\O_1\subset
\O_2$, then $\L(\O_1)\ge \L(\O_2)$.

Now, using a dilatation argument we can prove that
$\L(B_{R_1}(0))=\L(B_{R_2}(0))$ for all $0<R_1<R_2$. Hence we
conclude that $\L(\O)\equiv \bar{\L}$ does not depend of the
domain $\O$.

For $\phi\in \mathcal{C}^\infty_0(\ren)$, we set
$$
Q(\phi)\equiv \dfrac{\dint_{\re^N}\dint_{\re^N}
\dfrac{|\phi(x)-\phi(y)|^p}{|x-y|^{N+ps}}dxdy}{\dyle\irn\dfrac{|\phi(x)|^p}{|x|^{ps}}dx}.
$$
Let $\{\phi_n\}_n\subset \mathcal{C}^\infty_0(\ren)$ be such that
$Q(\phi_n)\to \L_{N,p,s}$. Without loss of generality and using a
symmetrization argument we can assume that
$\text{Supp}(\phi_n)\subset B_{R_n}(0)$. It is clear that
$Q(\phi_n)\ge \L(\text{Supp}(\phi_n))=\bar{\L}$, thus, as $n\to
\infty$, it follows that $\bar{\L}\le \L_{N,p,s}$. As a conclusion
we reach that $\bar{\L}=\L_{N,p,s}$ and the result follows.
\end{proof}

We need the next lemma.
\begin{Lemma}\label{lm:estim} Let $\O$ be
a bounded domain such that $0\in \O$. Assume that $u\in
W^{s,p}(\ren)$ is such that $u\ge 0$ in $\ren$, $u>0$ in $\O$ and
$L_{N,p,s}u\gneqq \l\dfrac{u^{p-1}}{|x|^{ps}}$ in $\O$, then there
exists $C>0$ such that $u(x)\ge C|x|^{-\g_1}$ in $B_\eta(0)$
where $\g_1$ is defined in Lemma \ref{lm01}.
\end{Lemma}
\begin{proof}
Without loss of generality we can assume that $B_1(0)\subset \O$.

Fixed $\l<\L_{N,p,s}$ and define
\begin{equation*}
\tilde{w}(x)= \left\{
\begin{array}{lll}
&|x|^{-\g_1}-1 &\mbox{  if }|x|<1,\\
& 0     &\mbox{   if }|x|>1. \end{array} \right. \end{equation*}
It is clear that $\tilde{w}\in W^{s,p}_0(B_1(0))$ and
\begin{equation}\label{www00}
\left\{
\begin{array}{rcl}
L_{p,s}\, \tilde{w} &=&h(x)\dfrac{\tilde{w}^{p-1}}{|x|^{ps}} \hbox{ in  } B_1(0),\\
\tilde{w}&=& 0\quad \hbox{  in  } \mathbb{R}^N\setminus B_1(0)
\end{array}
\right.
\end{equation}
where
$$
h(x)=\int_0^{\frac{1}{|x|}}|1-\s^{-\tilde{\g}}|^{p-2}(1-\s^{-\tilde{\g}})\s^{N-1}K(\s)d\s+(1-|x|^{\tilde{\g}})\int_{\frac{1}{|x|}}^\infty
\s^{N-1}K(\s)d\s.
$$
Using the definition of $\g_1$, see  Lemma \ref{lm01}, we can
prove that $h(x)\le \l$ for all $x\in B_1(0)$.

Since $L_{p,s}u\gneqq 0$ and $u>0$ in $\O$, then using the
nonlocal weak Harnack inequality in \cite{DI}, we get the
existence of $\e>0$ such that $u\ge \e$ in $\bar{B}_1(0)$.

Therefore we obtain that

\begin{equation}\label{www11}
\left\{
\begin{array}{lll}
L_{p,s}\, u &\ge & \l\dfrac{u^{p-1}}{|x|^ps}\mbox{  in   }B_1(0),\\
L_{p,s}\, \tilde{w} &\le & \l\dfrac{\tilde{w}^{p-1}}{|x|^{ps}}, \hbox{ in  } B_1(0),\\
u &\ge & \tilde{w}\quad \hbox{  in  } \mathbb{R}^N\setminus
B_1(0).
\end{array}
\right.
\end{equation}
Thus by the comparison principle in Lemma \ref{compaa}, it follows
that $\tilde{w}\le u$ which is the desired result.
\end{proof}

We are now in position to prove Theorem \ref{th:non}.

{\bf Proof of Theorem \ref{th:non}.} We argue by contradiction.
Assume the existence of $u\gneqq 0$ such that $u\in W^{s,p}(\ren)$
and $u$ is a supersolution to problem \eqref{prob} in $\O$, then
$u>0$ in $\O$. Let $\phi\in \mathcal{C}^\infty_0(B_\eta(0))$ with
$B_\eta(0)\subset\subset \O$ and $\eta>0$ to be chosen later.

Using Picone's inequality in Lemma \ref{pic}, it follows that
$$
||\phi||^p_{X^{s,p}_0(\ren)}\ge \int_{B_\eta(0)}
\dfrac{L_{p,s}(u)}{u^{p-1}}|\phi|^p dx.$$ Thus
$$
||\phi||^p_{X^{s,p}_0(\ren)}\ge \int_{B_\eta(0)}
u^{q-(p-1)}|\phi|^p dx.$$ Since $q>q_+(p,s)$, we get the existence
of $\e>0$ such that
$$
(\g_1-\e)(q-(p-1))>ps+\rho
$$
for some $\rho>0$. Thus, using Lemma \ref{lm:estim}, we can choose
$\eta>0$ such that
$$
u^{q-(p-1)}\ge C|x|^{-ps-\rho} \mbox{   in   } B_\eta(0).
$$
Therefore
$$
||\phi||^p_{X^{s,p}_0(\ren)}\ge
C\int_{B_\eta(0)}\dfrac{|\phi|^p}{|x|^{ps+\rho}} dx,$$ which is a
contradiction with the optimality of the Hardy inequality proved
in Lemma \ref{indep}. Hence we conclude. \cqd

\end{document}